\documentclass[a4paper,reqno]{amsart}
\usepackage{color}
\usepackage[normalem]{ulem}
\usepackage{amsmath}
\usepackage{amsthm}
\usepackage{amsfonts}
\usepackage{amssymb}
\usepackage{mathrsfs}
\usepackage[shortlabels]{enumitem}
\usepackage{graphicx}
\usepackage{subfigure}
\usepackage[dvipsnames]{xcolor}
\usepackage{tikz}
\usetikzlibrary{matrix}
\usepackage{mathtools} 
\usepackage{courier} 
\usepackage{hyperref}

\makeatletter                                                                   
\newlength{\bibsep}{\@listi \global\bibsep\itemsep \global\advance\bibsep by\parsep}
\makeatother  

\newtheorem{theorem}{Theorem}[section]
\numberwithin{theorem}{section}

\newtheorem{lemma}[theorem]{Lemma}

\theoremstyle{definition}
\newtheorem{definition}[theorem]{Definition}
\newtheorem{rem}[theorem]{Remark}

\numberwithin{equation}{section}

\newcommand{\N}{\mathbb{N}}

\newcommand{\R}{\mathbb{R}}

\newcommand{\C}{\mathbb{C}}

\newcommand\dom{\operatorname{dom}}

\DeclareMathOperator{\Div}{div}

\newcommand\cD{\mathcal D}
\newcommand\cE{\mathcal E}

\newcommand\cJ{\mathcal J}
\newcommand\cL{\mathcal L}
\newcommand\cM{\mathcal M}
\newcommand\cH{\mathcal H}
\newcommand\cK{\mathcal K}

\newcommand\fa{\mathfrak a}

\newcommand\sF{\mathscr F}

\newcommand\ov\overline

\newcommand\eps\varepsilon
\renewcommand\epsilon\varepsilon
\renewcommand\rho\varrho
\newcommand\al\alpha
\newcommand\la\lambda

\newcommand\ds\displaystyle

\newcommand\p\partial

\newcommand{\supp}{\operatorname{supp}}

\newcommand{\beq}{\begin{equation}}
\newcommand{\eeq}{\end{equation}}
\newcommand{\be}{\begin{equation*}}
\newcommand{\ee}{\end{equation*}}
\newcommand{\bmat}{\begin{pmatrix}}
\newcommand{\emat}{\end{pmatrix}}

\newcounter{counter_a}

\DeclarePairedDelimiter{\norma}{\lVert}{\rVert}

\newcommand{\diagdots}[3][-25]{%
  \rotatebox{#1}{\makebox[0pt]{\makebox[#2]{\xleaders\hbox{$\cdot$\hskip#3}\hfill\kern0pt}}}%
}

\author[D.~Buoso]{Davide Buoso}
\address{Dipartimento per lo Sviluppo Sostenibile e la Transizione Ecologica, Universit\`a degli Studi del Piemonte Orientale ``A.\ Avogadro'', piazza Sant'Eusebio 5,
Vercelli, 
Italy
}
\email{davide.buoso@uniupo.it}

\author[F.~Ferraresso]{Francesco Ferraresso}
\address{Dipartimento di scienze CFMN,
Via Vienna 2,
Sassari,
Italy
}
\email{fferraresso@uniss.it}

\date{\today}

\thanks{}
\usepackage[update,prepend]{epstopdf}

\title[]{Spectral convergence for the Reissner-Mindlin system in arbitrary dimension}



\begin{document}

\begin{abstract}
We establish the convergence of the resolvent of the Reissner-Mindlin system in any dimension $N \geq 2$, with any of the physically relevant boundary conditions, to the resolvent of the biharmonic operator with suitably defined boundary conditions in the vanishing thickness limit. Moreover, given a thin domain $\Omega_\delta$ in $\R^N$ with $1 \leq d < N$ thin directions, we prove that the resolvent of the Reissner-Mindlin system with free boundary conditions converges to the resolvent of a suitably defined Reissner-Mindlin system in the limiting domain $\Omega \subset \R^{N-d}$ as $\delta \to 0^+$. In both cases, the convergence is in operator norm, implying therefore the convergence of all the eigenvalues and spectral projections. In the thin domain case, we formulate a conjecture on the rate of convergence in terms of $\delta$, which is verified in the case of the cylinder $\Omega \times B_d(0, \delta)$. \\[0.2cm]

%
\end{abstract}

\maketitle


\section{Introduction}

The most widely considered models for the study of elastic thin plates through dimension reduction are the Kirchhoff-Love (shortly, KL) and the Reissner-Mindlin (shortly, RM) models. The former, which is also one of the first models developed in the literature (see e.g., \cite{MR0010851, MR0016009}) leads to the equation
\begin{equation}
\label{kleq}
\Delta^2 u=f,\quad\text{in\ }\Omega,
\end{equation}
where $f$ is the load on the three-dimensional plate $\Omega\times(-\epsilon/2,\epsilon/2)$ of negligible thickness $\epsilon$ and cross-section $\Omega\subseteq\mathbb R^2$, and the solution $u$ represents the vertical displacement from the equilibrium. Problem \eqref{kleq} has received a lot of attention both for applications and for its relevance in more theoretical questions also in higher dimensions; we refer to \cite{GazzGS} and the references therein.

As the KL model showed to be not accurate enough for applications, a better alternative was studied, namely the RM model (see e.g., \cite{BFbook91}), that leads to the system
\begin{equation}
\label{rmsys}
\begin{cases}
- \frac{\mu_1}{12} \Delta \beta - \frac{\mu_1 + \mu_2}{12} \nabla(\Div \beta) - \frac{\mu_1 k}{t^2} (\nabla w - \beta) = \frac{t^2}{12} F, &\textup{in $\Omega$,}\\
- \frac{\mu_1 k}{t^2} ( \Delta w - \Div \beta) = f, &\textup{in $\Omega$,}
\end{cases}
\end{equation}
where now the thickness $t>0$ of the plate $\Omega\times(-t/2,t/2)$ appears in the equations, $w$ represents the vertical displacement of the midplane, the vector field $\beta$ models the fibre rotation of the plate, $F$ is the applied couple, $\mu_1,\mu_2$ are the Lamé coefficients of the material, and $k>0$ is the so-called correcting term of the RM model. In the applied literature, it is classically assumed that the applied couple $F$ is zero \cite{MR1377485}. However, this is a rather restrictive assumption when considering the spectral convergence of the RM system; we shall therefore consider the more general form \eqref{rmsys} for $F \in L^2(\Omega)^N$ and $f \in L^2(\Omega)$. Note that the KL and the RM models are derived from the theory of elasticity under different sets of hypotheses; however, this difference can be summarised with the observation that, at least mathematically, the KL model has the additional assumption that $\beta=\nabla w$, which is nevertheless considered a good approximating assumption when the plate is extremely thin, i.e., when $t\to0^+$ (see e.g., \cite{BFbook91}). 

 Differently from problem \eqref{kleq}, problem \eqref{rmsys} has been investigated more deeply in the direction of applications, with much less literature devoted to more abstract questions (see e.g., \cite{MR1997951,MR3340708, BFbook91,MR3299138,destuynder}).

The present paper provides a rigorous spectral analysis for the RM problem \eqref{rmsys} in $\R^N$
in two particularly important cases: the asymptotic behaviour for $t\to0^+$, and the thin domain $\Omega_\delta$ where some dimensions disappear as $\delta \to 0^+$.

The behaviour of the RM problem \eqref{rmsys} in the limit case $t\to0^+$ has been largely investigated in the literature, in part because of its physical meaning but mainly for the nasty consequences of the so-called locking phenomenon, naming the difficulty of approximating numerically the solutions using classical techniques when the thickness $t$ of the plate becomes small (see \cite{MR3340708,MR3299138,MR1648387} for additional details). From a theoretic point of view, it is known that the RM problem \eqref{rmsys} converges to the KL problem \eqref{kleq} as $t\to0^+$ at least for clamped boundary conditions \cite[\SS VII.3]{BFbook91}; for other boundary conditions, asymptotic expansions of the solution were provided in \cite{MR1377485,MR1725748,MR1413859,MR1612135,MR2215994,MR1874448}.
However, these results are valid only for $\Omega\subseteq\mathbb R^2$, while there seems to be no result in higher dimension. Moreover, there are often additional restrictive assumptions on the data of the system \eqref{rmsys}, see for instance \cite[p.487]{MR1377485}, which are not in general justified when trying to establish the convergence of the resolvent operators. Here we show that the operator associated with the RM problem \eqref{Main problem: weak} converges to the bilaplacian \eqref{kleq}, in particular implying also spectral convergence for both eigenvalues and eigenprojectors (see Theorem \ref{thm:RMtoKLclamped}). Let us remark that our argument is not sensitive to the dimension $N$, and moreover allows us to treat all mathematically relevant boundary conditions (see Section \ref{Sec:system}).

While the convergence as $t\to 0^+$ is complicated because of a sudden change of the energy space in the limit, the convergence as $\delta \to 0^+$ for the RM system in a thin domain $\Omega_\delta\subset \R^N$ is involved because of the singular transformation that the geometry of $\Omega_\delta$ undergoes as $\delta \to 0^+$. We recall that a domain in $\R^N$ is called thin if one or more of their dimensions are very small compared to the others; in fact, it is usually assumed that these thin dimensions are so small to be negligible, inducing a dimension-reduction of the problem. Thin domains as limiting situations are interesting for applications specifically for approximations through asymptotics, since this procedure somehow simplifies the analysis, but provides a deep insight also in abstract settings, for the limit is singular and produces interesting examples and counterexamples. 

There is a quite vast literature on the thin domain analysis for (scalar) elliptic differential operators, see e.g., \cite{ferraresso_tri, ArrVil1, ArrVil2, MR4437918, ferraresso_bab, schatzman}.  In situations involving second-order linear partial differential operators and homogeneous Neumann boundary conditions, it is widely recognised that the problem's dimension can be effectively reduced by disregarding the narrow directions, as exemplified in \cite{HalRau} (see also \cite{Jim, JimKos}). Unfortunately, already for (scalar) fourth-order elliptic operators, the derivatives of the solutions along the thin directions give a non-trivial contribution in the limit \cite{AFL17, FerPro}, and therefore the classical approach of neglecting the thin directions cannot be applied.

Regarding systems of PDEs, the literature appears to be limited to the physically relevant cases ($N = 3$ and $d\in\{1,2\}$). Asymptotic expansions for thin elastic beams, plates, and shells in $\R^3$ have been established in the series of monographs \cite{MR4403921, MR4403922, MR4403923}. The elasticity system in thin three-dimensional beams has been considered for instance in \cite{CASADODIAZ2024127625}, while three-dimensional elastic multi-structures made of plates and beams are studied in \cite{COCV_2007__13_3_419_0}. Thin domain analysis for the system of linear elasticity has been also performed in \cite{MR1725748} and in \cite{MR1413859, MR1612135}; therein, the asymptotic expansions of eigenvalues and eigenfunctions of a thin three-dimensional cylindrical elastic clamped plate $\Omega^\eps = \omega \times (-\eps , \eps)$ have been found. 
We mention that related results in the homogenisation theory for periodically perforated elastic plates and composites can be found in \cite{OleBookElas, ZhiBookHom}.

In this paper we consider instead the RM system \eqref{rmsys} with free boundary conditions on thin domains in any dimension $N\geq 2$ and with any number $d$ of thin directions; assuming a certain technical assumption on the vertical avergages of the thin components of the fibre rotation $\beta_\delta$, we show that there is spectral convergence to an ``averaged'' RM operator in the limiting domain $\Omega \subset \R^{N-d}$, see Theorem \ref{main}. We emphasise that a technical assumption is required to ensure that the vertical components of $\beta_\delta$ do not give contributions to the limiting problem. In the literature, these problems are usually avoided because clamped boundary conditions are assumed on a part of the boundary; alternatively, it is usually required that the solutions satisfy some symmetry or that the data lie in a smaller space than $L^2$. Here, instead, we just assume that the integral average of the thin components of the data over each section is zero, for every $\delta$, see the definition of the space $\cH_\delta$ in Theorem \ref{main}; this is a rather natural assumption that rules out the possibility that the thin components of $\beta_\delta$ converge to constant functions in the limit.\\
We also remark that free boundary conditions are the only physical set of boundary conditions giving a non-trivial spectral limit. Indeed, the first eigenvalue (and then all the eigenvalues) $\lambda_1^{\rm RM}(\Omega_\delta)$ diverges as $\delta \to 0^+$ for all the boundary conditions with $w_\delta = 0$ on $\p \Omega_\delta$, in virtue of the Poincaré-type inequality
\begin{equation} \label{div}
\frac{C}{\delta^2} \leq \frac{\norma{\nabla w_\delta}_{L^2(\Omega_\delta)}^2}{\norma{w_\delta}_{L^2(\Omega_\delta)}^2};
\end{equation}
indeed, if $\sup_{\delta>0}\norma{w_\delta}_{L^2(\Omega_\delta)}^2 < \infty$, \eqref{div} implies that $\norma{\nabla w_\delta}_{L^2(\Omega_\delta)} \to + \infty$ as $\delta \to 0^+$, hence no spectral convergence is possible.



We recall that, in the literature, spectral convergence is usually obtained either via explicit computations using the Rayleigh quotient for the eigenvalues or via the so-called compact convergence of the resolvent operators in varying Hilbert spaces. In this paper, in addition to the compact convergence, we use for the first time the notion of generalised norm resolvent convergence (that is, the convergence in operator norm of the resolvents) in the sense of Vainikko, which is related to the generalised norm resolvent convergence introduced by Post, see \cite{PosBook} and \cite{MR3912880, MR4629855} for some recent applications in the case of the Laplacian. Note that Post norm resolvent convergence is equivalent to Weidmann/Stummel norm resolvent convergence (see \cite{Weibook1, Weibook2}) by virtue of the results in \cite{PosZim}. In concrete terms, let $(\cH_n)_n$, $\cH_0$ be a family of Hilbert spaces; if we denote by $(A_n)_n$ the closed operator associated to either problem \eqref{rmsys} or \eqref{proof:clampedRM} and by $A_0$ the limiting operator (either for the case $t\to0^+$ or for thin domains), we prove in Theorem \ref{cor:normconv} that 
\[
\norma{(A_n - \la)^{-1} \cE_n - \cE_n (A_0 - \la)^{-1}}_{\cL(\cH_0, \cH_n)} \to 0,
\] 
where $\cE_n: \cH_0 \to \cH_n$ is a sequence of asymptotically isometric operators, as in the usual setting by Vainikko (see for instance \cite{VainikkoSurvey}). Our proof of this result is based on operator theory considerations, and it is inspired by results by B\"ogli on the convergence of linear operators in varying Hilbert spaces, see e.g., \cite{Boegli2}. As a consequence of this generalised norm resolvent convergence, we prove that there exists $\omega(n) > 0$, $\omega(n) \to 0$ as $n\to \infty$, and a constant $C > 0$ such that
\[
\sum_{i=1}^m |\la^i_n - \la_0 | \leq C(|\la_0|) \omega(n),
\]
where $\la_0$ is an isolated eigenvalue of multiplicity $m$ of $A_0$, while $\la^i_n$ are (all) the eigenvalues of $A_n$ converging to $\la_0$ as $n \to +\infty$. This last fact had already been observed in \cite{IOS89}. For the thin domain analysis, when $\Omega_\delta = \Omega \times B_d(0,\delta)$, with $\Omega$ of class $C^2$, we are able to obtain the explicit estimate $\omega(\delta) \leq C\delta^{1/2}$ for some positive constant $C > 0$.  While our proof requires additional assumptions on the geometry and the regularity of the domain, we conjecture that the rate of convergence $\omega(\delta) = \delta^{1/2}$ is sharp and holds for all domains $\Omega_\delta$ as in \eqref{def:O_delta}. We remark that the rate of convergence $\omega(\delta)$ is in general worse than the rate of convergence of a simple eigenvalue $\la_j^\delta$ converging to $\la_0$, see for instance \cite{MR1297521}.

Let us observe that the lack of a uniform second Korn inequality (see Section \ref{Sec:prelim}) is one of the major difficulties that has to be overcome when dealing with the varying operators associated with the RM problem \eqref{rmsys}. This lack of uniform coercivity estimates seems to be a common trait of problems for systems of PDEs in varying domains; for instance, similar problems appear in the study of the Maxwell system since the solutions do not lie in $H^1(\Omega)^3$, see for instance \cite{FERRARESSO2023313}.


The paper is organised as follows. In Section 2 we provide some preliminaries, while in Section 3 we define the RM system and its various sets of boundary conditions. In Section 4 we prove the convergence to the KL model. Section 5 is dedicated to the convergence in the case of thin domains, and in Section 6 we obtain an estimate for the rate of convergence.


\section{Preliminaries and notation}
\label{Sec:prelim}

As we will be dealing with the convergence of operators acting on different Hilbert spaces, we need to recall the several notions of convergence in such a framework that were mainly developed by Stummel (see \cite{MR0291870, MR0291871, MR0410431}) and Vainikko (see \cite{MR0468159, MR0501887}).

\begin{definition} \label{def:Econv}
Let $\{\mathcal H_{\delta}\}_{\delta\in[0,\bar \delta)}$ be a family of Hilbert spaces. We assume that there exists a family of linear operators $\cE_\delta\in\mathcal L(\mathcal H_0,\mathcal H_\delta)$ called \emph{connecting system} such that, for all $u_0\in\cH_0$,
\begin{equation}\label{cond_ext}
\|\cE_\delta u_0\|_{\mathcal H_\delta}\rightarrow\|u_0\|_{\mathcal H_0}\,,\ \ \ {\rm as\ }\delta\rightarrow 0^+.
\end{equation}
\begin{enumerate}[label=(\roman*)]
\item Let $u_\delta\in \cH_\delta$. We say that $u_\delta$ $\cE$-converges to $u_0$ if $\|u_\delta-\cE_\delta u_0\|_{\cH_\delta}\rightarrow 0$ as $\delta\rightarrow 0^+$. We write $u_\delta\xrightarrow{\cE}u_0$.
\item Let $B_\delta\in\mathcal L(\cH_\delta)$. We say that $B_\delta$ $\cE\cE$-converges to  $B_0$ as $\delta\rightarrow 0^+$ if $B_\delta u_\delta\xrightarrow{\cE}B_0u_0$ whenever $u_\delta\xrightarrow{\cE}u_0$. We write $B_\delta\xrightarrow{\cE\cE}B_0$.
\item Let $B_\delta\in\mathcal L(\cH_\delta)$. We say that $B_\delta$ compactly converges to $B_0$ as $\delta\rightarrow 0^+$, and we write $B_\delta\xrightarrow{C}B_0$, if the following two conditions are satisfied
\begin{enumerate}[(a)]
\item $B_\delta\xrightarrow{\cE\cE}B_0$ as $\delta\rightarrow 0^+$;
\item for any family $u_\delta\in \cH_\delta$ such that $\|u_\delta\|_{\cH_\delta}=1$ for all $\delta\in(0,\bar \delta)$, there exists a subsequence $\{B_{\delta_k}u_{\delta_k}\}_{k\in\mathbb N}$ with $\delta_k\rightarrow 0^+$ as $k\rightarrow+\infty$, and $w_0\in\cH_0$ such that $B_{\delta_k}u_{\delta_k}\xrightarrow{\cE}w_0$ as $k\rightarrow+\infty$.
\end{enumerate}
\end{enumerate}
\end{definition}
Compact convergence of self-adjoint compact operators implies spectral convergence, as stated in the following theorem.
\begin{theorem}\label{sp_conv}
Let $A_\delta$, $\delta\in[0,\bar \delta)$ be a family of positive, self-adjoint differential operators on $\cH_\delta$ with domain $\cD(A_\delta)\subset \cH_\delta$. Assume moreover that
\begin{enumerate}[label=(\roman*)]
\item the resolvent operator $B_\delta:=A_\delta^{-1}$ is compact for all $\delta\in[0,\bar \delta)$;
\item $B_\delta\xrightarrow{C}B_0$ as $\delta\rightarrow 0^+$.
\end{enumerate}
Then, if $\lambda_0$ is an eigenvalue of $A_0$, there exists a sequence of eigenvalues $\lambda_\delta$ of $A_\delta$ such that $\lambda_\delta\rightarrow\lambda_0$ as $\delta\rightarrow 0^+$. Conversely, if $\lambda_\delta$ is an eigenvalue of $A_\delta$ for all $\delta\in(0,\bar \delta)$, and $\lambda_\delta\rightarrow\lambda_0$, then $\lambda_0$ is an eigenvalue of $A_0$. The generalised eigenspaces (resp. the spectral projections) of $A_\delta$ at $\lambda_0$ compactly converge to the $\la_0$-eigenspace (resp. the $\la_0$-spectral projection) of $A_0$ as $\delta \to 0^+$.
\end{theorem}
We refer to \cite[Thm. 4.10]{ACLC} and \cite[Thm. 4.2]{AFL17} for the proof of Theorem \ref{sp_conv}. We also refer to \cite[Prop. 2.6]{Boegli2} where a spectral convergence theorem is proved for sequences of closed operators with compact resolvent.\\[0.2cm]

Let us denote the ideal of compact operators between two Hilbert spaces $\cH_0$ to $\cH_1$ by ${\mathfrak S}_\infty(\cH_0, \cH_1)$. We recall now the notion of collective compactness as stated in \cite{AnsPalm}. 
\begin{definition} \label{def:collcomp}
Let $\cH_0$, $\cH_1$ be Hilbert spaces, and let $B$ be the unit ball in $\cH_0$. Let $F \subset {\mathfrak S}_\infty(\cH_0, \cH_1)$. We say that the family $\cK$ is collectively compact whenever the following set
\[
\cK B := \{ K x, \: K \in \cK, \: x \in B \}
\]
is precompact.
\end{definition}

Given a sequence $T_n$ of strongly convergent linear operators in $\cH_1$ (that is, there exists a linear operator $T \in \cL(\cH_1)$ such that $T_n x \to T x$ for all $x \in \cH_1$), and a collective compact family $\cK$ as in the previous proposition, it is easy to check that the composition $T_n \cK$ is a collective compact family, acting from $\cH_0$ to $\cH_1$. Similarly, given a sequence $S_n$ of strongly convergent linear operators in $\cH_0$, the composition $\cK S_n$ is collectively compact.

We proceed to state and prove two results that will be used in the sequel. They can be deduced from arguments similar to \cite[Prop. 2.10, Prop. 2.13]{Boegli2}, with the help of \cite[Thm 3.4]{AnsPalm}. However, since we are using these results for operators acting between possibly different Hilbert spaces, and in a slightly different setting from both Stummel discrete convergence and Anselone and Palmer collectively compact convergence, we prefer to provide full details (see also Section \ref{Sec:normresconv} where we use these results in the framework of Vainikko norm resolvent convergence).

\begin{theorem} \label{thm:collcompadj}
Assume that $(T_n)_n$, $T_n \in {\mathfrak S}_\infty(\cH_0, \cH_1)$, $n \in \N$, $(T_n)_n$ is collectively compact, and $T_n$ converges strongly to a bounded operator $T_0 \in \cL(\cH_0, \cH_1)$. If in addition $(T_n)_n$ has the property that 
\begin{equation} \label{eq:continuityprop}
u_n \rightharpoonup u \:\: \textup{in $\cH_0$} \:\: \Rightarrow  \ \exists (u_{n_k})_k \subset (u_n)_n, \:\: T_{n_k} u_{n_k} \to T_0 u,
\end{equation}
then $(T^*_n)_n$ converges strongly to $T_0^* \in \cL(\cH_1, \cH_0)$ and $(T^*_n)_n$ is collectively compact.
\end{theorem}
\begin{proof}
We divide the proof in two steps.\\
\textbf{Step 1.} $(T_n)^*$ converges strongly to $T^*_0$ as $n \to \infty$. \\
Since $T_n$ converges strongly to $T_0$, given $f_1 \in \cH_1$ and $f_0 \in \cH_0$, 
\[
(T_n^* f_1, f_0)_{\cH_0} = (f_1, T_n f_0)_{\cH_1} \to (f_1, T_0 f_0)_{\cH_1} = (T_0^* f_1, f_0)_{\cH_0} ,
\]
hence $T_n^* f_1 \rightharpoonup T_0^* f_1$ in $\cH_0$. Now, 
\[
\norma{T_n^* f_1}_{\cH_1}^2 = (T_n^* f_1, T_n^* f_1)_{\cH_1} = (T_n T_n^* f_1, f_1)_{\cH_1}.
\]
Since $(T_n)_n$ is collectively compact and $(T_n^* f_1)_n$ is weakly convergent (hence, bounded), the sequence $(T_n T_n^* f_1)_n$ has a convergent subsequence, that is there exists $u \in \cH_1$ and a subsequence $(T_{n_k})_k$ such that, $T_{n_k} T_{n_k}^* f_1 \to u$, as $k \to \infty$. We need to prove that $u = T_0 T_0^* f_1$.\\ 
Set $T_{n_k}^* f_1 = v_{n_k}$, $v_{n_k}\rightharpoonup v_0 := T_0^* f_1$.  We have $T_{n_k} v_{n_k} \to u$. By property \eqref{eq:continuityprop}, up to taking another subsequence, we have $T_{n_k}v_{n_k} \to T_0 v_0 = T_0 T_0^* f_1$, and therefore we have established that  $u = T_0 T_0^* f_1$.
%
We now conclude that
\[
\norma{T_{n_k}^* f_1}_{\cH_1}^2 = (T_{n_k} T_{n_k}^* f_1, f_1)_{\cH_1} \to (T_0 T_0^* f_1, f_1)_{\cH_1} = \norma{T_0^* f_1}_{\cH_1}^2,
\]
and therefore, $T^*_{n_k}$ converges strongly to $T_0^*$ as $k \to + \infty$. It is not difficult to check that in fact we also have $T^*_n$ converges strongly to $T_0^*$ (since for every subsequence of $T^*_n$ we can extract a further subsequence converging to $T_0^*$). Thus, $T^*_n$ converges strongly to $T_0^*$, concluding Step 1.\\[0.15cm]
\noindent \textbf{Step 2.} $(T_n^*)_n$ is a collectively compact family in $\cH_1$. \\
Since $(T_n)_n$ is collectively compact and $T_n^*$ is strongly convergent, the composition $(T_n T^*_n)_n$ is collectively compact. Assume that $(u_n)_n \in \cH_1$, $\sup_n \norma{u_n}_{\cH_1} \leq C < + \infty$ and let us prove that $(T_n^* u_n)_n$ has a convergent subsequence in $\cH_1$. We may assume without loss of generality that, up to a subsequence, $u_n \rightharpoonup u \in \cH_1$ as $n \to \infty$.  Since $T_n T^*_n$ is collectively compact, up to a subsequence, there exists $w \in \cH_1$ such that $T_n T^*_n u_n \to w$. But since $T_n^*$ converges strongly to $T_0^*$ by Step 1, and $T_n$ converges strongly to $T_0$ by assumption, the composition $T_n T^*_n$ is strongly convergent to $T_0 T^*_0$. Therefore,
\[
(T_n T^*_n u_n, \psi)_{\cH_1} = (u_n, T_n T^*_n \psi)_{\cH_1} \to (u, T_0 T_0^* \psi)_{\cH_1} = (T_0 T_0^* u, \psi)_{\cH_1}
\]
for all $\psi \in \cH_1$, or equivalently, $T_n T^*_n u_n \rightharpoonup T_0 T_0^* u$. By uniqueness of the weak limit it must then be $w = T_0 T_0^* u$. Finally, we note that
\[
\norma{T^*_n u_n}^2_{\cH_1} = (T_n T^*_n u_n, u_n)_{\cH_1} \to  (T_0 T^*_0 u, u)_{\cH_1} = \norma{T_0^* u}^2_{\cH_1},
\]
so $T^*_n u_n \to T^*_0 u_0$ up to a subsequence. This concludes the proof of Step 2 and of the Theorem.
\end{proof}

\begin{theorem} \label{cor:normconv}
Let $(T_n)_n$ be as in Theorem \ref{thm:collcompadj}. Then $T_0$ is compact and the following norm convergence holds $\norma{T_n - T_0}_{\cL(\cH_0, \cH_1)} \to 0$, as $n \to + \infty$.
\end{theorem}
\begin{proof}
Since $T_n$ converges strongly to $T_0$, $(T_n)_n$ is collectively compact, and the strong closure of a collectively compact family is collectively compact (see \cite[Prop 2.1]{AnsPalm}), we deduce immediately that $T_0$ is compact. \\
Since $(T_n^*)_n$ is collectively compact and $(T_n)_n$ is strongly convergent, $(T_n^* T_n)_n$ is collectively compact. Property \eqref{eq:continuityprop} now implies that 
\begin{equation} \label{eq:T^*T}
v_n \rightharpoonup v_0 \: \: \textup{in $\cH_0$} \quad \Rightarrow \quad \exists (v_{n_k})_k \subset (v_n)_n \:\: : \: \: T_{n_k}^* T_{n_k} v_{n_k} \to T_0^* T_0 v_0.
\end{equation}
%
%
%
Since $(T_n)_n$ is a collectively compact sequence, $\sup_{n \in \N}\norma{T_n}_{\cL(\cH_0, \cH_1)} < +\infty$; otherwise, if for a contradiction there exists a sequence $u_n$, $\norma{u_n} = 1$ but $\norma{T_n u_n} \geq n$, then up to a subsequence $u_n \rightharpoonup u_0$, and by collective compactness, $(T_n u_n)_n$ is precompact, hence in particular totally bounded, a contradiction to $\norma{T_n u_n} \to + \infty$. We need to show that $\norma{T_n - T_0}_{\cL(\cH_0, \cH_1)} \to 0$. Without loss of generality we may assume that there exists $u_n \in \cH_0$, $\norma{u_n} = 1$ such that
\[
\norma{T_n - T_0}_{\cL(\cH_0, \cH_1)} = \norma{(T_n - T_0) u_n}_{\cH_1} + \frac{1}{n},
\]
for all sufficiently large $n$. Up to a subsequence, we may assume that $u_n \rightharpoonup u_0 \in \cH_0$. First note that since $T_0$ is compact, $T_0 u_n \to T_0 u_0$ in $\cH_0$. Further observe that
\begin{multline} \label{proof:Tn-T0}
\norma{(T_n- T_0) u_n}_{\cH_1}^2 = (T_n^* T_n u_n, u_n)_{\cH_0} - (T_n u_n, T_0 u_n)_{\cH_1} \\
 - (T_0 u_n, T_n u_n)_{\cH_1} + \norma{T_0 u_n}^2_{\cH_1};
\end{multline}
up to a subsequence, we might assume that $T_n^* T_n u_n \to T_0^* T_0 u_0$ and $T_n u_n \to T_0 u_0$ due to \eqref{eq:T^*T} and \eqref{eq:continuityprop}. We conclude that, up to a subsequence, the right-hand side of \eqref{proof:Tn-T0} tends to zero as $n \to \infty$. Therefore, we conclude that there exists a subsequence $(n_k)_k$ such that $\norma{T_{n_k} - T_0}_{\cL(\cH_0, \cH_1)} \to 0$ as $k\to +\infty$. But then arguing exactly as in the end of Step 1 in the proof of Theorem \ref{thm:collcompadj} the whole sequence has to tend to zero. Finally, $\norma{T_n - T_0}_{\cL(\cH_0, \cH_1)} \to 0$, concluding the proof of the Theorem.
\end{proof}

\begin{rem}
Note that in Theorem \ref{thm:collcompadj} we do not assume any selfadjointness or normality assumption of the operators $T_n$, that would not even make sense in this setting since each $T_n$ is acting between different Hilbert spaces. If $\cH_0 =\cH_1$ and the operators are normal, Theorem \ref{cor:normconv} was proved in \cite[Prop. A.3]{BeKh22} (see also \cite{AnsPalm} for the bounded case).
\end{rem}

The Hilbert spaces $\cH_\delta$ we will be mainly dealing with in the sequel will be Sobolev spaces defined on bounded open sets $\Omega$ of $\mathbb R^N$. In particular, for any $k\in\mathbb N$, we will denote with $H^k(\Omega)$ the Sobolev space  of functions $u\in L^2(\Omega)$ with all weak  derivatives up to order $k$ in $L^2(\Omega)$. The space $H^k(\Omega)$ is  endowed with the scalar product
\[
\langle u, v\rangle_{H^k(\Omega)}:=\int_{\Omega} \sum_{|\alpha|\le k}\partial^\alpha u \: \partial^\alpha v \,dx\,,\ \ \ \forall u,v\in H^k(\Omega),
\]
where $\partial^\alpha u=\partial^{\alpha_1}\cdots\partial^{\alpha_N}u$ denotes the partial derivative of $u$ of order $\alpha$ for any multiindex $\alpha\in\mathbb N^{N}$. We will also denote by $H^k_0(\Omega)$ the closure in $H^k(\Omega)$ of the space  $C^\infty_c(\Omega)$ of smooth functions with compact support.

Given a vector field $\eta \in H^1(\Omega)^N$ we denote by $D \eta$ its Jacobian matrix and by $\eps(\eta)$ the symmetric part of $D \eta$, that is 
\begin{equation}
\label{epsilon}
\eps(\eta) = \frac{1}{2} \big( D \eta + D^T \eta\big), \qquad \eps(\eta)_{ij} = \frac{1}{2} \bigg( \frac{\p \eta^j}{\p x_i} + \frac{\p \eta^i}{\p x_j} \bigg).
\end{equation}
As we will see in Section \ref{Sec:system}, the weak formulation of problem \eqref{rmsys} is given via a sesquilinear form involving $\eps(\eta)$ and $\Div \eta$. In general, the boundedness of $\eps(\eta)$ and $\Div \eta$ in $L^2(\Omega)$ is not enough to ensure that $D \eta$ is bounded in $L^2(\Omega)$. However, this is true if $\eta \in H^1_0(\Omega)^N$; as a consequence of the divergence theorem, 
\[
2\norma{\eps(\eta)}^2_{L^2(\Omega)} = \norma{D \eta}^2_{L^2(\Omega)} + \norma{\Div \eta}^2_{L^2(\Omega)},
\]
for all $\eta \in H^1_0(\Omega)^N$; we then obtain the classical (first) Korn inequality:
\[
\norma{D \eta}_{L^2(\Omega)} \leq \sqrt{2} \norma{\eps(\eta)}_{L^2(\Omega)}\,,
\]
for all $\eta \in H^1_0(\Omega)^N$. More in general, a Korn-type inequality remains true for vector fields in $H^1(\Omega)$, provided they have positive distance from the subspace of rigid translations. We have indeed the following classical result. 
\begin{theorem}
\label{thm: general Korn}
Let $\Omega$ be a bounded domain of $\,\R^N\!$ with Lipschitz boundary and let $\eta \in H^1(\Omega)^N$. Let $V$ be a weakly closed linear space of $H^1(\Omega)^N$ such that $V \cap \mathfrak{M} = \emptyset$, where
$$\mathfrak{M} = \{ \Psi(x) \in \mathcal{L}(\R^N, \R^N) : \Psi(x) = A x + b, \textup{$A=-A^T\in M_{N\times N}(\mathbb R)$, $b \in \R^N$} \}. $$
Assume that $\eta \in V$. Then
\begin{equation}
\label{Korn's ineq}
\int_{\Omega} |D \eta|^2 dx \leq C \int_\Omega |\eps(\eta)|^2,
\end{equation}
where the constant $C$ depends only on $\Omega$.
\end{theorem}
\begin{proof}
We refer to \cite[Theorem 2, \S2]{KonOlei}.
\end{proof}

If $\Omega$ is a bounded Lipschitz domain of $\,\R^N\!$, the second Korn inequality states that any vector field $\eta \in L^2(\Omega)^N$ with $\eps(\eta) \in L^2(\Omega)$, belongs to $H^1(\Omega)^N$, and there exists a constant $C>0$ (depending on $\Omega$) such that the inequality
\begin{equation}\label{secondKorn}
\int_{\Omega} |D \eta|^2 dx \leq C \bigg( \int_\Omega |\eps(\eta)|^2 + |\eta|^2 \bigg)
\end{equation}
is valid. We refer to \cite{MR2119999} for a proof of \eqref{secondKorn} and for an interesting discussion on the different Korn inequalities and their relation with the distributional Saint Venant equality. A constructive proof of \eqref{secondKorn} can be found in \cite{MR1978636}.

\section{The Reissner-Mindlin system}
\label{Sec:system}
Let $\Omega$ be a bounded and Lipschitz domain of $\R^N$, $N \geq 2$, and let $V,W$ be closed spaces such that $H^1_0(\Omega)^{N}\subseteq V \subseteq H^1(\Omega)^{N}$ and $H^1_0(\Omega)\subseteq W \subseteq H^1(\Omega)$.
The Reissner-Mindlin eigenvalue problem is given by the following weak formulation
\begin{equation}
\label{Main problem: weak}
\mathcal{R}_t\left[(\beta,w),(\eta,v)\right] = \lambda \int_\Omega\left(wv + \frac{t^2}{12}\beta\cdot\eta\right)dx,
\end{equation}
for all $(\eta, v) \in V\times W$, where $(\beta, u) \in V\times W$ is the eigenvector and $\lambda\in\mathbb R$ is the eigenvalue. Here the sesquilinear form $\mathcal R_t$ is defined by
$$
\mathcal R_t\left[(\beta,w),(\eta,v)\right]=\fa(\beta, \eta) + \frac{Ek}{2(1+\sigma)t^2} \int_\Omega(\nabla w - \beta)\cdot( \nabla v - \eta)dx,
$$
where $\fa(\cdot, \cdot)$ from $H^1(\Omega)^N \times H^1(\Omega)^N$ to $\R$ is the elliptic bilinear form defined by
\begin{equation}
\label{def: a(,)}
\fa(\beta, \eta) = \frac{E}{12(1-\sigma^2)} \int_{\Omega} \Big((1-\sigma) \eps(\beta) : \eps(\eta) + \sigma \Div(\beta) \Div(\eta)\Big)\, dx,
\end{equation}
where $\eps(\cdot)$ is the linear strain tensor defined in \eqref{epsilon}, $E > 0$ is the Young modulus, and $\sigma \in \big(\!- \frac{1}{N-1}, 1\big)$ is the Poisson ratio. Recall that the Lam\'e coefficients $\mu_1$ and $\mu_2$ are related to $E$ and $\sigma$ via 
\[
\mu_1 = \frac{E}{2(1+\sigma)}, \quad \mu_2 = \frac{\sigma E}{2(1-\sigma^2)}.
\]
We note that the bilinear form $\fa \geq 0$ has a non-trivial kernel, which consists of all the rigid motions $A x + b$, with $A \in {\rm Sym}_N(\R)$, $A^T = - A$, $b \in \R$, in $\R^N$. 
Nevertheless, the second Korn inequality \eqref{secondKorn} shows that the form $\mathfrak{a}$ is coercive in $H^1(\Omega)^N$, that is, there exist positive constants $M$ and $C$ such that
\[
C \norma{\Phi}^2_{H^1(\Omega)^N} \leq \mathfrak{a}(\Phi, \Phi) + M \norma{\Phi}^2_{L^2(\Omega)^N}
\]
for all $\Phi \in \dom(\mathfrak{a})$, provided that $\sigma \in \big(\!- \frac{1}{N-1}, 1\big)$.

By integration by parts (see \eqref{BC equation}), one can prove that the classical formulation of Problem \eqref{Main problem: weak} is given by
\begin{equation}
\label{Main problem: strong}
\begin{cases}
- \frac{\mu_1}{12} \Delta \beta - \frac{\mu_1 + \mu_2}{12} \nabla(\Div \beta) - \frac{\mu_1 k}{t^2} (\nabla w - \beta) = \lambda \frac{t^2}{12} \beta, &\textup{in $\Omega$,}\\
- \frac{\mu_1 k}{t^2} ( \Delta w - \Div \beta) = \lambda w, &\textup{in $\Omega$,}
\end{cases}
\end{equation}
where the boundary conditions will depend on the particular choice of the space $V$. 

We now describe various possible boundary conditions for the Reissner-Mindlin system. We recall that some of them have already been discussed in the literature for their physical relevance (see e.g., \cite{MR1377485}). We will use the notation $(\Phi)_{\p \Omega}$ to denote the tangential trace of $\Phi \in H^1(\Omega)^N$ on $\p \Omega$, that is $(\Phi)_{\p \Omega} = (\Phi - (\Phi \cdot \nu) \nu)|_{\p \Omega}$, where $\nu$ is the outer normal to $\p \Omega$.
An integration-by-parts procedure yields the following identity
{\small
\begin{multline}
\label{BC equation}
\fa(\beta, \eta) + \frac{\mu_1 k}{t^2} (\nabla w - \beta, \nabla v - \eta)\\
=- \frac{\mu_1}{12}\int_{\Omega} \Delta \beta \cdot \eta
- \frac{\mu_1+\mu_2}{12} \int_\Omega \nabla (\Div \beta) \cdot \eta 
- \frac{\mu_1 k}{t^2} \int_\Omega (\nabla w - \beta) \cdot \eta dx 
- \frac{\mu_1 k}{t^2} \int_{\Omega} (\Delta w - \Div \beta) v dx\\
+  \frac{\mu_1}{6} \int_{\partial \Omega} \Big(\eps(\beta)\cdot\nu\Big)_{\partial\Omega}\cdot (\eta)_{\partial\Omega} dS 
+ \int_{\partial \Omega} \Big(\frac{\mu_1}{6}\frac{\partial\beta}{\partial\nu}\cdot \nu+\frac{\mu_2}{12}\Div \beta \Big)(\eta \cdot \nu) dS 
 + \frac{\kappa}{t^2} \int_{\partial \Omega} ((\nabla w - \beta) \cdot \nu) v\, dS,
\end{multline}
}
for all $\eta \in V$, $v \in W$. As usual, since equality \eqref{BC equation} is valid for any choice $(\eta,v)\in C^\infty_c(\Omega)^{N+1}$, we get that the solutions must satisfy the system of equations \eqref{Main problem: strong} in $\Omega$, while the respective boundary conditions depend on the choice of the subspaces $V, W$. \\
The interesting choices for $W$ are $H^1(\Omega)$ and $H^1_0(\Omega)$, while there are four classical choices for $V$: $H^1_0(\Omega)^{N}$, $H^1(\Omega)^{N}$, $\{\Phi \in (H^1(\Omega))^N : \Phi \cdot \nu = 0 \: \textup{on $\partial \Omega$} \}$, and $ \{\Phi \in (H^1(\Omega))^N : (\Phi)_{\p \Omega} = 0 \: \textup{on $\partial \Omega$} \}$, summing up to a total of eight possible sets of boundary conditions, five of which have significant physical interest (see \cite{MR1377485}).\\

\noindent\textbf{Hard clamped boundary conditions}\\
In this case $V = (H^1_0(\Omega))^{N}$, $W=H^1_0(\Omega)$, producing the following
\[
\begin{cases}
\beta = 0 = w, &\textup{on $\partial \Omega$}.
\end{cases}
\]

\noindent\textbf{Soft clamped boundary conditions}\\
In this case $V = \{\Phi \in (H^1(\Omega))^N : \Phi \cdot \nu = 0 \: \textup{on $\partial \Omega$} \}$, $W = H^1_0(\Omega)$. Here the boundary integral involving $\eta \cdot \nu$ vanishes, producing the following
\[
\begin{cases}
(\eps(\beta)\nu)_{\p \Omega} = 0, &\textup{on $\partial \Omega$,}\\
w=0=\beta \cdot \nu, &\textup{on $\partial \Omega$.}\\
\end{cases}
\]

\noindent \textbf{Hard simply supported boundary conditions}\\
In this case $V = \{\Phi \in (H^1(\Omega))^N : (\Phi)_{\p \Omega} = 0 \: \textup{on $\partial \Omega$} \}$, $W = H^1_0(\Omega)$. Here the boundary integral presenting $(\eta)_{\p\Omega}$ vanishes, and since
$$
\frac{\mu_1}{6}\frac{\partial\beta}{\partial\nu}\cdot \nu+\frac{\mu_2}{12}\Div \beta=
\frac{E}{12(1-\sigma^2)}\left((1-\sigma)\frac{\partial\beta}{\partial\nu}\cdot \nu+\sigma\Div \beta\right),
$$
we have the following
\[
\begin{cases}
(1-\sigma) \frac{\partial\beta}{\partial\nu}\cdot \nu+ \sigma \Div \beta = 0, &\textup{on $\partial \Omega$,}\\
w = 0 = (\beta)_{\p \Omega}, &\textup{on $\partial \Omega$.}
\end{cases}
\]

\noindent \textbf{Soft simply supported boundary conditions}\\
In this case $V = (H^1(\Omega))^N$, $W = H^1_0(\Omega)$. Here all the boundary integral presenting $\beta$ do not vanish, so they produce the following
\[
\begin{cases}
(1-\sigma) \frac{\partial\beta}{\partial\nu}\cdot \nu+ \sigma \Div \beta = 0, &\textup{on $\partial \Omega$,}\\
(\eps(\beta) \nu)_{\p \Omega}= 0, &\textup{on $\partial \Omega$,}\\
w = 0, &\textup{on $\partial \Omega$.}
\end{cases}
\]

\noindent\textbf{(Free) Neumann boundary conditions}\\
In this case $V = H^1(\Omega)^N$, $W=H^1(\Omega)$. Then all the boundary integrals in \eqref{BC equation} are non-vanishing, hence producing
\[
\begin{cases}
(1-\sigma) \frac{\partial\beta}{\partial\nu}\cdot \nu+ \sigma \Div \beta = 0, &\textup{on $\partial \Omega$,}\\
(\eps(\beta) \nu)_{\p \Omega} = 0,   &\textup{on $\partial \Omega$,}\\
\frac{\partial w}{\partial \nu} - \beta \cdot \nu = 0, &\textup{on $\partial \Omega$.}
\end{cases}
\]

Apart from these classical boundary conditions, we define three additional sets of boundary conditions, to which we attach names following the same criterion.

\noindent \textbf{Hard rigid boundary conditions}\\
In this case $V = H^1_0(\Omega)^N$, $W = H^1(\Omega)$. Here the boundary integrals presenting $\eta$ vanish, producing the following
\[
\begin{cases}
\frac{\partial w}{\partial \nu}=0=\beta, &\textup{on $\partial \Omega$.}
\end{cases}
\]

\noindent \textbf{Soft rigid boundary conditions}\\
In this case $V = \{\Phi \in (H^1(\Omega))^N : \Phi \cdot \nu = 0 \: \textup{on $\partial \Omega$} \}$, $W = H^1(\Omega)$. Here all the boundary integral presenting $\beta$ do not vanish, so they produce the following
\[
\begin{cases}
(\eps(\beta)\nu)_{\p \Omega} = 0, &\textup{on $\partial \Omega$,}\\
\frac{\partial w}{\partial \nu}=0=\beta \cdot \nu, &\textup{on $\partial \Omega$.}\\
\end{cases}
\]

\noindent\textbf{Weak Neumann boundary conditions}\\
In this case $V = \{\Phi \in (H^1(\Omega))^N : (\Phi)_{\p \Omega} = 0 \: \textup{on $\partial \Omega$} \}$, $W=H^1(\Omega)$. This produces
\[
\begin{cases}
(1-\sigma) \frac{\partial\beta}{\partial\nu}\cdot \nu+ \sigma \Div \beta = 0, &\textup{on $\partial \Omega$,}\\
(\beta)_{\p \Omega}=0, &\textup{on $\partial \Omega$,}\\
 \frac{\partial w}{\partial \nu} - \beta \cdot \nu = 0, &\textup{on $\partial \Omega$.}\\
\end{cases}
\]

We conclude by observing that a pair $(\beta,w)$ can belong to the kernel of problem \eqref{Main problem: weak} only if
$$
\beta= a,\quad w=a\cdot x+b,\quad a\in\mathbb R^N,b\in\mathbb R,
$$
meaning that the Reissner-Mindlin problem with either hard rigid, soft rigid, or weak Neumann boundary conditions will have a one-dimensional kernel $\{(0,b):b\in\mathbb R\}$, while imposing free Neumann boundary conditions will give a $(N+1)$-dimensional kernel. In all the other cases the kernel is trivial.


\section{Convergence of the Reissner-Mindlin to the Kirchhoff-Love plate model}
In this section we give a dimension-independent argument of the convergence of the eigenvalues of the Reissner-Mindlin system $(\la_k(t))_k$ to the eigenvalues of the bilaplacian $\Delta^2$ as $t\to0^+$.

We remark that this convergence was proved in dimension $N = 2$ via a rather complicated regularity argument which heavily relies on the existence of a suitable Helmholtz-type decomposition of the vector-field $\beta$ that cannot be easily extended to higher dimensions (see \cite{BFbook91}). Here instead we will develop a spectral convergence argument that uses in a critical way the Stummel-Vainikko theory (see Section \ref{Sec:prelim}). We stress the fact that our argument yields immediately not just the convergence of the eigenvalues but also that of the generalised eigenspaces (and in particular of the spectral projections) in the $L^2$-topology.

Consider the following formulation of the Reissner-Mindlin eigenvalue problem
\begin{equation}
\label{RMid}
\begin{cases}
- \frac{\mu_1}{12} \Delta \beta - \frac{\mu_1 + \mu_2}{12} \nabla(\Div \beta) - \frac{\mu_1 k}{t^2} (\nabla w - \beta) +\frac{t^2}{12} \beta= \lambda \frac{t^2}{12} \beta, &\textup{in $\Omega$,}\\
- \frac{\mu_1 k}{t^2} ( \Delta w - \Div \beta) +w = \lambda w, &\textup{in $\Omega$,}\\
\end{cases}
\end{equation}
complemented with any of the boundary conditions listed in Section \ref{Sec:system}. We will prove that \eqref{RMid} converges, as $t\to 0^+$, to the problem
\begin{equation}
\label{clampedKL}
\frac{E}{12(1-\sigma^2)} \Delta^2 u+u=\lambda u,\quad \textup{in $\Omega$,}
\end{equation}
complemented with suitable boundary conditions (see Remark \ref{bcbl} below). Notice that the weak formulation of problem \eqref{RMid} is
\begin{equation}
\label{Main problem: weak_bis}
\mathcal{R}_t\left[(\beta,w),(\eta,v)\right] +\int_\Omega\left(wv + \frac{t^2}{12}\beta\cdot\eta\right)dx= \lambda \int_\Omega\left(wv + \frac{t^2}{12}\beta\cdot\eta\right)dx,
\end{equation}
and in particular the eigenvalues of \eqref{Main problem: weak_bis} coincide with the shifted eigenvalues of \eqref{Main problem: weak}. Similarly, the weak formulation of problem \eqref{clampedKL} is
\begin{equation}
\frac{E}{12(1-\sigma^2)}\int_{\Omega}\left((1-\sigma)D^2u:D^2v+\sigma\Delta u\Delta v\right)dx+\int_\Omega uv dx= \lambda \int_\Omega uv dx.
\end{equation}
Notice also that problems \eqref{RMid} and \eqref{clampedKL} always have a trivial kernel.

In order to use the Stummel-Vainikko theory we need to define a suitable connecting system between the limiting Hilbert space $\cH_0 = L^2(\Omega)$ and the varying Hilbert spaces $\cH_t \simeq L^2(\Omega)^{N+1}$, endowed with the norm
\[
\norma{(\beta_t, w_t)}_{\cH_t} := \sqrt{{\small \frac{t^2}{12}} \norma{\beta_t}^2_{L^2(\Omega)^N} + \norma{w_t}^2_{L^2(\Omega)}}.
\]
As the extension operator between $\cH_0$ and $\cH_t$, we choose the natural identification
\[
\cE_t(u_0) = (0,u_0), \qquad u_0 \in \cH_0.
\]
It is trivial to check that $\norma{\cE_t(u_0)}_{\cH_t} = \norma{u_0}_{\cH_0}$, implying that $((\cH_t)_t, \cH_0, (\cE_t)_t)$ is an admissible connecting system. Note in passing that if $(\Phi_t, \varphi_t)$ is $\cE$-convergent to $\varphi_1$, and also to $\varphi_2$, with respect to the connecting system $((\cH_t)_t, \cH_0, (\cE_t)_t)$, we deduce that
\[
\norma{(\Phi_t, \varphi_t) - \cE_t(\varphi_1)}^2_{\cH_t} = t^2/12\norma{\Phi_t}^2_{L^2(\Omega)^N} + \norma{\varphi_t - \varphi_1}^2_{L^2(\Omega)} \to 0,
\]
implying that $\norma{\varphi_t - \varphi_1}^2_{L^2(\Omega)} \to 0$. Moreover,
\[
\norma{(\Phi_t, \varphi_t) - \cE_t(\varphi_2)}^2_{\cH_t} = t^2/12\norma{\Phi_t}^2_{L^2(\Omega)^N} + \norma{\varphi_t - \varphi_2}^2_{L^2(\Omega)} \to 0,
\]
so $\norma{\varphi_t - \varphi_2}^2_{L^2(\Omega)} \to 0$. Combining the two we deduce that $\varphi_1 = \varphi_2$, so the $\cE$-limit is unique. 

Let $A_t$ be the closed differential operator acting in the Hilbert space $\cH_t$ as
\[
A_t(\beta_t, w_t) := 
\begin{pmatrix}
- \frac{\mu_1}{12} \Delta \beta_t - \frac{\mu_1 + \mu_2}{12} \nabla(\Div \beta_t) - \frac{\mu_1 k}{t^2} (\nabla w_t - \beta_t) +\frac{t^2}{12} \beta\\
- \frac{\mu_1 k}{t^2} ( \Delta w_t - \Div \beta_t) +w\\
\end{pmatrix}
\]
with $\dom(A_t) = \{ (U_t, u_t) \in \cH_t : (U_t, u_t)\in V\times W , A_t(U_t, u_t) \in \cH_t \}$, where $V,W$ are any of the spaces discussed in Section \ref{Sec:system}. \\
Let $A_0$ be the closed differential operator acting in $\cH_0$ as 
\[
A_0(u) := \frac{E}{12(1-\sigma^2)} \Delta^2 u+u
\]
with $\dom(A_0) = \{ u_0 \in \cH_0 : u_0\in W, \nabla u_0\in V , A_0(u_0) \in \cH_0 \}$.

\begin{rem}\label{bcbl}
The choice of the boundary conditions for the Reissner-Mindlin system \eqref{RMid} (i.e., the choice of $V,W$) leads to different boundary conditions at the limit for the bilaplacian:\\
$\bullet$ hard and soft clamped conditions converge to clamped (Dirichlet) boundary conditions:
\[u=\frac{\partial u}{\partial\nu}=0;\]
$\bullet$ hard and soft simply supported conditions converge to hinged (Navier) boundary conditions:
\[u=(1-\sigma)\frac{\partial^2 u}{\partial\nu^2} + \sigma \Delta u=0;\]
$\bullet$ soft rigid conditions converge to intermediate (Kuttler-Sigillito) boundary conditions: 
\[
\frac{\partial u}{\partial\nu}=\frac{\partial \Delta u}{\partial\nu}+(1-\sigma)\Div_{\partial\Omega}(\nu\cdot D^2 u)_{\partial\Omega}=0;
\]
$\bullet$ free Neumann conditions converge to free (Neumann) boundary conditions: 
\[
(1-\sigma)\frac{\partial^2 u}{\partial\nu^2} + \sigma \Delta u =\frac{\partial \Delta u}{\partial\nu}+(1-\sigma)\Div_{\partial\Omega}(\nu\cdot D^2 u)_{\partial\Omega}=0.
\]
We refer to \cite{BuoKen} for a comprehensive discussion of boundary conditions for the bilaplacian (see also \cite{MR3571822}). We also note that the remaining two sets of boundary conditions, namely hard rigid and weak Neumann conditions, lead to non-standard boundary value problems. They are associated with a self-adjoint operator with compact resolvent; as customary, if the domain $\Omega$ is only Lipschitz the two problems have to be understood in the weak sense. If instead $\Omega$ is smooth enough (\emph{e.g.}, $\partial \Omega\in C^2$) then it is possible to characterise more explicitly the boundary conditions. In particular, hard rigid boundary conditions give the following boundary conditions for the biharmonic operator in the limit 
$$
\begin{cases}
\frac{\partial u}{\partial\nu}=0 \qquad \text{on\ }\partial\Omega,\\
u \text{\ constant on every connected component of\ }\partial \Omega,
\end{cases}
$$
while weak Neumann boundary conditions give
$$
\begin{cases}
(1-\sigma)\frac{\partial^2 u}{\partial\nu^2} + \sigma \Delta u=0 \quad\text{on\ }\partial\Omega,\\
u \text{\ constant on every connected component of\ }\partial \Omega.
\end{cases}
$$
\end{rem}

\begin{theorem} \label{thm:RMtoKLclamped}
Let $B_t = A_t^{-1}$ and $B_0 = A_0^{-1}$. Then $B_t \xrightarrow{C} B_0$ as $t \to 0^+$. In particular the eigenvalues $\la_k(t) \in \sigma(A_t)$ converge to $\la_k(0) \in \sigma(A_0)$ as $t \to 0^+$.
\end{theorem}
\begin{proof}
\textbf{Step 1: a priori bounds and weak convergence.} Let $(F_t, f_t) \in \cH_t$ be a uniformly bounded sequence in $\cH_t$. We might assume without loss of generality that $t F_t \rightharpoonup F_0$ in $L^2(\Omega)^N$ and $f_t \rightharpoonup f_0$ in $L^2(\Omega)$. We consider the problem 
\begin{equation}
\label{proof:clampedRM}
\begin{cases}
- \frac{E}{24(1+\sigma)} \Delta \beta_t - \frac{E}{24(1-\sigma)} \nabla(\Div \beta_t) - \frac{\mu_1 k}{t^2} (\nabla w_t - \beta_t) + \frac{t^2}{12} \beta_t= \frac{t^2}{12}F_t, &\textup{in $\Omega$,}\\
- \frac{\mu_1 k}{t^2} ( \Delta w_t - \Div \beta_t) +w_t= f_t, &\textup{in $\Omega$,}
\end{cases}
\end{equation}
complemented with the appropriate boundary conditions depending on the choice of $V$ and $W$. From this we obtain the a priori bound
\[
\begin{split}
&\frac{E}{24(1+\sigma)}\norma{\eps(\beta_t)}^2 + \frac{E}{24(1-\sigma)} \norma{\Div \beta_t}^2 + \frac{\mu_1 k}{t^2} \norma{\nabla w_t - \beta_t}^2 \\
&\hspace{5cm}+\frac{t^2}{12}\norma{\beta_t}^2 + \norma{w_t}^2 \leq \sup_{t > 0} \norma{(F_t, f_t)}_{\cH_t}^2 < \infty.
\end{split}
\]
This bound readily implies that $(w_t)_t$ and $(\nabla w_t-\beta_t)_t$ are $L^2$-bounded. We claim that also
$$
\sup_{t \in (0,1)} \norma{\beta_t}_{L^2(\Omega)}<+\infty.
$$
Arguing for a contradiction, we may choose a sequence $(t_k)_k$ such that $\norma{\beta_{t_k}}_{L^2(\Omega)}=z_{t_k}\to+\infty$. Considering now
$$
\tilde\beta_{t_k}= \frac{\beta_{t_k}}{z_{t_k}},\quad \tilde w_{t_k}= \frac{w_{t_k}}{z_{t_k}},
$$
we have $\norma{\tilde\beta_{t_k}}_{L^2(\Omega)}=1$, $\norma{\tilde w_{t_k}}_{L^2(\Omega)}\to0$; furthermore, $(\tilde\beta_{t_k},\tilde w_{t_k})$ solves 
\begin{equation}
\begin{cases}
- \frac{E}{24(1+\sigma)} \Delta \tilde\beta_{t_k} - \frac{E}{24(1-\sigma)} \nabla(\Div \tilde\beta_{t_k}) - \frac{\mu_1 k}{t_k^2} (\nabla \tilde w_{t_k} - \tilde\beta_{t_k}) + \frac{t_k^2}{12} \tilde\beta_{t_k}= \frac{t_k^2}{12}\tilde F_{t_k}, &\textup{in $\Omega$,}\\
- \frac{\mu_1 k}{t_k^2} ( \Delta \tilde w_{t_k} - \Div \tilde\beta_{t_k}) +\tilde w_{t_k}= \tilde f_{t_k}, &\textup{in $\Omega$,}
\end{cases}
\end{equation}
where $\tilde F_{t_k}=F_{t_k}/z_{t_k}$ and $\tilde f_{t_k}=f_{t_k}/z_{t_k}$. This now implies that
$$
\norma{\eps(\tilde{\beta}_{t_k})},\  \norma{\nabla w_{t_k} - \tilde{\beta}_{t_k}} \to 0.
$$
In particular,
$$
\sup_{k}\norma{\tilde{\beta}_{t_k}}_{H^1(\Omega)}\le C \sup_{k}\left(\norma{\eps(\tilde{\beta}_{t_k})}_{L^2(\Omega)}+\norma{\tilde{\beta}_{t_k}}_{L^2(\Omega)}\right)<+\infty,
$$
where $C$ is the constant given by Korn's inequality (see Theorem \ref{thm: general Korn}). This means that there exists $\tilde\beta_0$ such that $\tilde\beta_{t_k}\rightharpoonup\tilde\beta_0$ in $H^1(\Omega)^N$, hence $\tilde\beta_{t_k}\to\tilde\beta_0$ in $L^2(\Omega)^N$. Since
\[
\begin{split}
&\norma{\tilde w_{t_k}}^2_{H^1(\Omega)}=\norma{\tilde w_{t_k}}^2_{L^2(\Omega)}+\norma{\nabla \tilde w_{t_k}}^2_{L^2(\Omega)} \\
&\hspace{1cm} \le \norma{\tilde w_{t_k}}^2_{L^2(\Omega)}+\norma{\nabla \tilde w_{t_k}-\tilde \beta_{t_k}}^2_{L^2(\Omega)}+\norma{\tilde \beta_{t_k}}^2_{L^2(\Omega)},
\end{split}
\]
then, up to a subsequence, $(\tilde w_{t_k})_k$ has a weak $H^1$-limit that must be zero since $\norma{\tilde w_{t_k}}^2_{L^2(\Omega)}\to0$.\\
Since $\norma{\nabla \tilde w_{t_k}-\tilde \beta_{t_k}}^2_{L^2(\Omega)}\to0$, we conclude that $\tilde\beta_0=0$ as well. However, the fact that $\norma{\tilde\beta_{t_k}}_{L^2(\Omega)}=1$ provides the contradiction.

Thus $(\beta_t)_t$ is a uniformly bounded sequence in $H^1(\Omega)^N$, whence there exists a vector field $\beta_0 \in H^1(\Omega)^N$ such that, up to a subsequence, $\beta_t \rightharpoonup \beta_0$ in $H^1(\Omega)^N$, strongly in $L^2(\Omega)^N$. Moreover, since $\big(\frac{\nabla w_t - \beta_t}{t}\big)_t$ is uniformly bounded in $L^2(\Omega)^N$, then $\norma{\nabla w_t - \beta_t}_{L^2(\Omega)} \to 0$, implying that $\nabla w_t \to \beta_0$ in $L^2(\Omega)^N$. Therefore the sequence $(w_t)_t$ is uniformly bounded in $H^1(\Omega)$ 
so, up to a subsequence, there exists $w_0 \in H^1(\Omega)$ such that $w_t \rightharpoonup w_0$ in $H^1(\Omega)$, strongly in $L^2(\Omega)$. Now
\[
\nabla w_t - \beta_t \to 0, \qquad \beta_t \to \beta_0, \qquad \nabla w_t \rightharpoonup \nabla w_0,
\]
imply $\nabla w_0 = \beta_0 \in V\subseteq H^1(\Omega)^N$, and in particular $w_0 \in H^2(\Omega)$. \\[0.2cm]


\textbf{Step 2: finding the limiting problem.} For any given $(\eta, v) \in V\times W$ we consider the weak formulation of \eqref{proof:clampedRM}, that is
\begin{equation}
\label{proof:clampedRMweak}
\begin{split}
&\frac{E}{24(1+\sigma)} \int_{\Omega} \nabla \beta_t : \nabla \eta + \frac{E}{24(1-\sigma)} \int_{\Omega}\Div \beta_t \Div \eta \\
&\hspace{0.2cm}+ \frac{\mu_1 k}{t^2}\int_{\Omega} (\nabla w_t - \beta_t) (\nabla v - \eta) +\int_\Omega\bigg(\frac{t^2}{12} \beta_t \eta + w_t v\bigg)= \int_{\Omega} \bigg(\frac{t^2}{12} F_t \eta + f_t v\bigg).
\end{split}
\end{equation}

We first choose $\eta = t\Phi$ for $\Phi\in H^1_0(\Omega)^N$ and $v = 0$. From \eqref{proof:clampedRMweak} we get
\[
\mu_1 k\int_{\Omega} \frac{(\nabla w_t - \beta_t)}{t} \cdot \Phi\to 0,
\]
implying that any subsequence of $\frac{(\nabla w_t - \beta_t)}{t}$ tends weakly to zero in $L^2(\Omega)^N$. 

Now let $v_0 \in W$ be a function such that $\nabla v_0\in V$, and chose $\eta = \nabla v_0$, $v=v_0$ as test functions in \eqref{proof:clampedRMweak}. We obtain
\begin{multline}
\label{proof:clampedRMweak2}
 \frac{E}{24(1+\sigma)} \int_{\Omega} \nabla \beta_t : D^2 v_0 + \frac{E}{24(1-\sigma)} \int_{\Omega}\Div \beta_t \Delta v_0 \\
 +\int_\Omega\bigg(\frac{t^2}{12} \beta_t \nabla v_0 + w_t v_0\bigg)= \int_{\Omega} \bigg(\frac{t^2}{12} F_t \nabla v_0 + f_t v_0\bigg).
\end{multline}
Due to the previous steps of the proof $\nabla \beta_t \rightharpoonup D^2 w_0$, $\Div \beta_t \rightharpoonup \Delta w_0$, $t^2 F_t \to 0$, so passing to the limit in \eqref{proof:clampedRMweak2} we deduce
\[
 \frac{E}{24(1+\sigma)} \int_{\Omega} D^2 w_0 : D^2 v_0 + \frac{E}{24(1-\sigma)} \int_{\Omega}\Delta w_0 \Delta v_0 +\int_\Omega w_0 v_0= \int_{\Omega} f_0 v_0,
\]
and since this equality holds for all functions $v_0 \in \{v\in W: \nabla v\in V\}$, we deduce that $w_0 \in W$, $\nabla w_0 \in V$, and it solves 
\begin{equation}
\label{KLPoi}
\frac{E}{12(1-\sigma^2)} \Delta^2 w_0 + w_0 = f_0,\quad \textup{in $\Omega$.}
\end{equation}


\textbf{Step 3: compact convergence.} According to the definition of compact convergence we must check that, given a sequence of data $(F_t, f_t) \in \cH_t$ which is $\cE$-convergent to $f_0 \in \cH_0$, we have 
\begin{equation}
\label{eq:BtB0}
\norma{B_t (F_t, f_t) - \cE_t B_0 f_0}^2_{\cH_t} \to 0.
\end{equation}
Let $B_t(F_t, f_t) = (\beta_t, w_t)$ and $u_0 = B_0 f_0$. We can rewrite \eqref{eq:BtB0} as
\[
\frac{t^2}{12}\norma{\beta_t}^2_{L^2(\Omega)} + \norma{w_t - w_0}^2_{L^2(\Omega)} \to 0.
\]
Due to Step 2, we have that $w_t \rightharpoonup w_0$ in $H^1(\Omega)$, strongly in $L^2(\Omega)$, and $w_0$ solves problem \eqref{clampedKL}. \\
For the second condition in the compact convergence definition, we only assume that $(F_t, f_t) \in \cH_t$ is a uniformly bounded sequence in $\cH_t$, as done at the beginning of Step 1. We must check that $(B_t (F_t, f_t))_t$ is a $\cE$-precompact sequence, that is, it has a $\cE$-converging subsequence. Step 1 and Step 2 of the proof immediately imply that $B_t(F_t, f_t) = (\beta_t, w_t)$ $\cE$-converges to $w_0 \in \cH_0$, where $w_0$ solves \eqref{KLPoi}. This concludes the proof.
\end{proof}

\begin{rem} \label{rem:simplify}
Step 1 in the proof of Theorem \ref{thm:RMtoKLclamped} can be significantly simplified in all the cases where either $w = 0$ or $\frac{\p w}{\p \nu} =0$ on $\p \Omega$. In the first case, since $(w_t)_t$ and $(\Delta w_t)_t$ are uniformly bounded, we conclude that
\begin{equation} \label{eq:w}
\norma{\nabla w_t}^2 = (- \Delta w_t, w_t) \leq C ( \norma{w_t}^2 + \norma{\Delta w_t}^2 ),
\end{equation}
and therefore $(\nabla w_t)_t$ is uniformly bounded in $t$. Since also $\frac{\mu_1 k}{t^2} \norma{\nabla w_t - \beta_t}^2$ is uniformly bounded, $ \norma{\nabla w_t - \beta_t} \to 0$ as $t \to 0^+$, hence from $\norma{\beta_t} \leq \norma{\nabla w_t - \beta_t} + \norma{w_t}$ we conclude that $(\beta_t)_t$ is uniformly bounded in $L^2$. From Korn's inequality we then conclude that $(\beta_t)_t$ is uniformly bounded in $H^1$. \\
The exact same argument works when $\frac{\p w_t}{\p \nu} = 0$ on $\p \Omega$, because \eqref{eq:w} still holds true, as a consequence of
\[
\norma{\nabla w_t}^2 + (\Delta w_t, w_t) = \langle \p_\nu w_t, w_t \rangle_{-1/2, 1/2} = 0.
\]
\end{rem}

In the proof of Theorem \ref{thm:RMtoKLclamped} we proved that $A_t^{-1}\cE_t f_0 - \cE_t A_0^{-1} f_0 \to 0$ for all $f_0 \in \cH_0$. One could ask if this convergence happens uniformly with respect to the datum $f_0 \in \cH_0$. Let us the introduce the following definition, which is the natural extension of the Vainikko theory to the convergence in norm of the operators.

\begin{definition} \label{def:normresconv}
Let $(\cH_n)_n, \cH_0$ be Hilbert spaces and let $((\cH_n)_n, \cH_0, (\cE_n)_n)$ be a connecting system in the sense of Vainikko. Let $(A_n)_n$ be a family of closed linear operators, $A_n$ acting in $\cH_n$ for each $n$, $A_0$ acting in $\cH_0$. Let $\la \in \big(\bigcap_n \rho(A_n) \cap \rho(A_0)\big)$. We say that $A_n$ converges to $A_0$ with respect to (Vainikko) generalised norm resolvent convergence if
\[
\sup_{f_0 \in \cH_0} \frac{\norma{(A_n - \la)^{-1} \cE_n f_0 - \cE_n (A_0 - \la)^{-1} f_0}_{\cH_n}}{\norma{f_0}_{\cH_0}} \to 0
\]
as $n \to \infty$. 
\end{definition}

\begin{theorem}\label{thm:normresconvt}
$A_t$ converges to $A_0$ in generalised norm resolvent convergence as $t \to 0^+$.
\end{theorem}
\begin{proof}
Let us set for simplicity $B_t := A_t^{-1}$, $t \geq 0$. Due to Theorem \ref{thm:RMtoKLclamped}, we already know that for fixed $f_0 \in \cH_0$ 
\begin{equation} \label{proof:convt}
\norma{(B_t \cE_t - \cE_t B_0) f_0}^2_{\cH_t} = \frac{t^2}{12} \norma{\beta_t}^2_{L^2(\Omega)^N} + \norma{w_t - w_0}^2_{L^2(\Omega)} \to 0,
\end{equation}
where as in the proof of Thm \ref{thm:RMtoKLclamped} we wrote $B_t \cE_t f_0 = (\beta_t, w_t) \in L^2(\Omega)^{N+1}$. Define now $\tilde{B}_t g := \frac{t}{\sqrt{12}} B_t g$, $g \in \cH_t$. Then we can rewrite 
\eqref{proof:convt} as 
\[
\norma{(\tilde{B}_t \cE_t - \cE_t B_0) f_0}^2_{L^2(\Omega)^{N+1}} \to 0.
\]
We also note that the dependence on $t$ for $\cE_t$ is trivial; therefore, with a little abuse of notation we will write $\cE$ instead $\cE_t$ for the operator acting between $\cH_0$ and $\cH_t$, $\cE(v) = (0, v)$, for all $v \in \cH_0$. \\
Let $T_t := (\tilde{B}_t \cE - \cE B_0)$. Then $(T_t)_t$ is a collectively compact family of operators acting from $\cH_0$ to $L^2(\Omega)^{N+1}$, $T_t$ is compact for every $t$, and therefore we are almost in position to apply Theorem \ref{cor:normconv}. In order to do so, we need to check that \eqref{eq:continuityprop} is satisfied by $(T_t)_t$ with $T_0 = 0$. \\
For this, let $(u_t)_t$ be a bounded sequence in $L^2(\Omega)$, and we might assume that $u_t \rightharpoonup u$ in $L^2(\Omega)$. Since $B_0$ is compact we clearly have $B_0 u_t \to B_0 u$ as $t\to 0^+$. Moreover, from the proof of Theorem \ref{thm:RMtoKLclamped}, possibly up to a subsequence,
\[
\lim_{t \to 0^+} \tilde{B}_t \cE u_t = \cE B_0 ({\rm w-}\!\!\lim_{t\to 0^+} u_t) = \cE B_0 u.
\]
Therefore 
\[
\lim_{t \to 0^+} T_t u_t = \lim_{t \to 0^+} (\tilde{B}_t \cE u_t - \cE B_0 u_t) = \cE B_0 u - \cE B_0 u = 0,
\]
and so $(T_t)_t$ has property \eqref{eq:continuityprop}. By Theorem \ref{cor:normconv} we conclude that $T_t \to 0$ in $\cL(\cH_0, L^2(\Omega)^{N+1})$, concluding the proof. 
\end{proof}

\begin{rem}
As a consequence of Theorem \ref{thm:normresconvt}, there exists $\omega(t) > 0$, $\omega(t) \to 0^+$ as $t \to 0^+$ such that
\begin{equation} \label{eq:rateconvt}
\norma{(B_t \cE_t - \cE_t B_0) f_0}_{\cH_t} \leq \omega(t) \norma{f_0}
\end{equation}
for all $f_0 \in \cH_0$. Precise estimates on $\omega(t)$ in terms of powers of $t$ are available only in very specific cases, usually $N =2$, $\Omega$ either convex or of class $C^2$, and only for clamped boundary conditions, see for instance \cite{MR1648387}, where they establish the sharp bound $\omega(t) = t$, by using the a priori estimates contained in \cite{MR1025088}. Note that such a priori estimates hold only for clamped boundary conditions where one has access to a specific well-adapted Helmholtz decomposition in $H^1_0(\Omega)^2$; we have already noted in Remark \ref{rem:simplify} that uniform a priori estimates are much more complicated in all the cases where $w_t$ is not zero at the boundary. As far as we know, the only other available results on the asymptotic behaviour of the Reissner-Mindlin system are contained in \cite{MR1377485}; however, the asymptotic expansions of the eigenvalues and eigenfunctions therein are obtained under the assumption that the so-called applied couple $F$ (the datum of the first equation in \eqref{rmsys}) is zero. Therefore, those asymptotics cannot be immediately applied to deduce a rate of convergence of the resolvent operators. Further note that any bound as in \eqref{eq:rateconvt} can be used to deduce rates of convergence for the eigenvalues, see Theorem \ref{thm:rateeigen} below.
\end{rem}


\section{Spectral convergence in the thin domain}
\label{Sec: shrink channel}

In this section we consider the Reissner-Mindlin system on a family of shrinking domains $\Omega_\delta \subset \R^N$; the whole family $(\Omega_\delta)_{\delta \in (0,1]}$ will be called \emph{thin domain}. 
Let $\delta \in \R$ be a small positive parameter. Let $N \geq 2$ and let $d \in\{1, \dots, N-1\}$. Let $\Omega$ be a smooth bounded open set in $\R^{N - d}$. Define
\begin{equation}\label{def:Omega_1}
\Omega_1 := \{ (x',y') \in \R^{N - d} \times \R^d : x' \in \Omega, \: y' \in \Omega^{II}_1(x') \},
\end{equation}
where $\Omega \subset \{(x',0) \in \R^{N-d} \times \{0\}^d\}$ and the sections $\Omega^{II}_1(x')$ are bilipschitz-equivalent to the unit ball $B_d(0,1)$. That is, setting for simplicity $B=B_d(0,1)$, we assume that for each $x' \in \Omega$ there exists a $C^{0,1}$-diffeomorphism $L_{x'}$ such that
\begin{equation} \label{eq:L_x}
L_{x'} : B \to \Omega_1^{II}(x').
\end{equation}
Moreover, we define $L : \Omega \times B \to \Omega_1$ by $L(x',y') = (x', L_{x'}(y'))$ and we assume that $L$ is a $C^{0,1}$-diffeomorphism as well. We further define
\begin{equation}
\label{def:O_delta}
\Omega_\delta = \{(x', \delta y')\in\R^{N - d} \times \R^d : (x',y') \in \Omega_1 \}.
\end{equation}
Denote by $g(x') = |\Omega_1^{II}(x')|$, the $d$-dimensional Lebesgue measure of the section $\Omega_1^{II}(x')$. In the standing assumptions, $g$ is a Lipschitz function and that there exists a positive constant $c_0 > 0$ such that $g \geq c_0$. 
Note that $\Omega_\delta$ collapses to the limit set $\Omega \times 0^d$ as $\delta \to 0^+$. For instance, when $d = 1$, the thin domain $\Omega_\delta$ has the more explicit representation
\[
\Omega_\delta = \{ (x,y)\in\R^{N - 1} \times \R: x\in \Omega, -\delta f_1(x) < y < \delta f_2(x) \},
\]
for some positive Lipschitz functions $f_1$, $f_2$.

In this section we focus on the Reissner-Mindlin system with free boundary conditions, which has a nontrivial kernel that can cause some issues since the associated operator is not invertible. This situation can be easily overcome by adding a multiple of the identity, which has the effect of shifting all the eigenvalues while keeping the eigenspaces invariant. Hence, in passing to the limit as $\delta \to 0^+$, we will deal with the following problem (here $t>0$ will be fixed)
\begin{equation}
\label{bvp-R_delta}
\begin{cases}
- \frac{\mu_1}{12} \Delta \beta - \frac{\mu_1 + \mu_2}{12} \nabla(\Div \beta) - \frac{\mu_1 k}{t^2} (\nabla w - \beta) + \frac{t^2}{12} \beta= \lambda \frac{t^2}{12} \beta, &\textup{in $\Omega_\delta$,}\\
- \frac{\mu_1 k}{t^2} ( \Delta w - \Div \beta) + w = \lambda w, &\textup{in $\Omega_\delta$,}\\
(1-\sigma)(\frac{\partial\beta}{\partial\nu}\cdot \nu) + \sigma \Div \beta = 0, &\textup{on $\partial \Omega_\delta$,}\\
(\eps(\beta) \nu)_{\p \Omega_\delta}= 0,   &\textup{on $\partial \Omega_\delta $,}\\
(\nabla w - \beta) \cdot \nu = 0, &\textup{on $\partial \Omega_\delta.$}
\end{cases}
\end{equation}

For any $\delta \in (0,1]$ we define the partial differential operator $A_\delta$ associated with problem \eqref{bvp-R_delta}  by
\begin{equation} \label{def:A_delta}
A_\delta \begin{pmatrix} U \\ u \end{pmatrix} = \begin{pmatrix} - \frac{\mu_1}{12} \Delta U - \frac{\mu_1 + \mu_2}{12} \nabla(\Div U) - \frac{\mu_1 k}{t^2} (\nabla  u- U) + \frac{t^2}{12} U \\ - \frac{\mu_1 k}{t^2} ( \Delta u - \Div U) + v \end{pmatrix}
\end{equation}
with domain
\[
\dom(A_\delta) = \{ (U, u) \in L^2(\Omega_\delta)^{N+1} : \: A_\delta (U,u) \in L^2(\Omega)^{N+1} \}.
\]
Note that for any fixed $\delta$, we have $\dom(A_\delta) \subset H^1(\Omega_\delta)$ by the second Korn inequality \eqref{secondKorn}.

We define the operator $A_0$ as
{\small\begin{equation}\label{def:A_0}
\begin{split}
&A_0\begin{pmatrix} \Phi \\ \varphi \end{pmatrix} =\\
& \begin{pmatrix} 
\frac{1}{g} \left[-\frac{E}{24(1+ \sigma)} \Div(g \nabla\Phi) - \frac{E(1 + (d + 1) \sigma)}{24 (1 + \sigma)(1 + (d-1) \sigma)} \nabla( g \Div \Phi) - \frac{E k}{2(1-\sigma)t^2} (\nabla\varphi - \Phi) + \frac{t^2}{12} \Phi \right] \\
\frac{1}{g} \left[ -\frac{E k}{2(1-\sigma)t^2} \Div(g (\nabla \varphi - \Phi)) + \varphi \right]
\end{pmatrix}
\end{split}
\end{equation}
}
with domain
\[
\dom(A_0) = \{ (\Phi, \varphi) \in L^2(\Omega, g(x) dx)^{N-d+1}: \: A_0(\Phi,\varphi) \in L^2(\Omega)^{N-d+1}\}.
\]
For $\delta \in (0,1]$, let us define the operator $\cM_\delta : L^2(\Omega_\delta) \to L^2(\Omega, g(x) dx)$ by
\[
(\cM_\delta f)(x) = \frac{1}{|\Omega_\delta^{II}(x)|} \int_{\Omega_\delta^{II}(x)} f(x,y) dy
\]
for almost all $x \in \Omega$. With a little abuse of notation we will use the same symbol $\cM_\delta$ to denote the operator acting from $L^2(\Omega_\delta)^k$ to $L^2(\Omega, g(x) dx)^k$ for some $k \geq 1$, with the understanding that the action of the averaging operator is componentwise.

The family of operators $A_\delta$ define naturally a family of Hilbert spaces $\cH_\delta = L^2(\Omega_\delta)^{N+1}$, and it is natural to pass to the limit as $\delta\to0^+$.\\[0.3cm]

\begin{theorem} \label{main}
Let $\Omega_\delta$, $A_\delta$, $A_0$ be defined as in \eqref{def:O_delta}, \eqref{def:A_delta}, \eqref{def:A_0}. Set
\[
\cH_\delta = L^2(\Omega_\delta; \delta^{-d} dx dy)^{N - d} \times \delta L_0^2(\Omega_\delta; \delta^{-d} dx dy )^d \times L^2(\Omega_\delta; \delta^{-d} dx dy),
\]
endowed with the usual norm $\norma{V}_{\cH_\delta} = \norma{V}_{L^2(\Omega_\delta; \delta^{-d} dx dy)^{N+1}}$, where we have used the notation 
$$L^2_0(\Omega_\delta)^d = \{ U \in L^2(\Omega)^d : \cM_\delta U = 0  \}.$$ Set moreover
\[
\cH_0 = L^2(\Omega;\, g(x) dx)^{N-d + 1}.
\]
Let $B_\delta = A_\delta^{-1}$, $B_0 = A_0^{-1}$. Define the linear operator $\cE_\delta \in \cL(\cH_0, \cH_\delta)$ by
\[
(\cE_\delta u_0)(x,y) = (u^1_0, \dots, u_0^{N-d}, 0, \dots, 0, u_0^{N-d + 1})(x).
\]
Then $B_\delta \xrightarrow{C}B_0$ as $\delta\rightarrow 0^+$ with respect to the connecting system $(\cH_0, (\cH_\delta)_\delta, (\cE_\delta)_\delta)$.
\end{theorem}
\begin{proof}
We identify a point $P \in \Omega_\delta \subset \R^N$ with the ordered couple $(x,y) \in \R^N$, $x \in \R^{N - d}$, $y \in \R^d$, where $y$ plays the role of the thin variable. Define the diffeomorphism $\Phi_\delta: \Omega_\delta \to \Omega_1$, $(x,y) \mapsto \Phi_\delta(x,y) = (x, y/\delta)$. Set
$$\cJ=\cJ_\delta := (D \Phi_\delta) =
\begin{pmatrix}
1 & 0 \\
0 & 1/\delta
\end{pmatrix}.
$$
Note that if $U \in  L^2(\Omega_\delta; \delta^{-d/2} dx dy)^{N - d} \times \delta L^2(\Omega_\delta; \delta^{-d/2} dx dy )^d$, then $\cJ^T U \circ \Phi_\delta^{-1} \in L^2(\Omega_1)^N$. In the setting of the compact convergence with respect to the connecting system $(\cH_0, \cH_\delta, \cE_\delta)$, we have to consider the following Poisson problem with data $G_\delta := (F_\delta, f) \in \cH_\delta$, with $\norma{G_\delta}_{\cH_\delta} \leq 1$,
\begin{equation}
\label{bvp-Poi}
\begin{cases}
- \frac{E}{24(1+\sigma)} \Delta \beta - \frac{E}{24(1-\sigma^2)} \nabla(\Div \beta) - \frac{E k}{2 (1+\sigma) t^2} (\nabla w - \beta) + \frac{t^2}{12}\beta = \frac{t^2}{12} F_\delta, &\textup{in $\Omega_\delta$,}\\
- \frac{E k}{2t^2 (1+\sigma)} ( \Delta w - \Div \beta) + w = f, &\textup{in $\Omega_\delta$,}\\
\frac{E}{12(1+\sigma)} (\frac{\partial\beta}{\partial\nu}\cdot \nu) + \frac{\sigma E}{24(1-\sigma^2)} (\Div \beta) = 0, &\textup{on $\partial \Omega_\delta$,}\\
(\eps(\beta) \nu)_{\p \Omega_\delta}= 0,   &\textup{on $\partial \Omega_\delta$,} \\
(\nabla w - \beta) \cdot \nu = 0, &\textup{on $\partial \Omega_\delta.$}
\end{cases}
\end{equation}

For $\eta \in H^1(\Omega_\delta, \delta^{-d}dxdy)^{N-d} \times \delta H^1(\Omega_\delta, \delta^{-d}dxdy)^d$, $v \in H^1(\Omega_\delta, \delta^{-d}dx dy)$, the weak formulation of problem \eqref{bvp-Poi} is
\begin{equation} \label{eq:weakPoi}
\begin{split}
\frac{E}{12(1-\sigma^2)} &\int_{\Omega_\delta} \bigg[ (1-\sigma) \eps(\beta_\delta) : \eps(\eta) + \sigma \Div \beta_\delta \Div \eta \bigg] \delta^{-d} dx dy \\
&+ \int_{\Omega_\delta} \bigg[\frac{Ek}{2(1-\sigma^2)t^2} (\nabla w_\delta - \beta_\delta) (\nabla v - \eta) + \frac{t^2}{12} \beta_\delta \cdot \eta + w v \bigg] \delta^{-d} dx dy \\
&= \int_{\Omega_\delta} \bigg[ \frac{t^2}{12} F_\delta \cdot \eta + f v \bigg] \delta^{-d} dx dy.
\end{split}
\end{equation}

To shorten the notation, given a vector field $V \in L^2(\Omega_\delta)^m$, $m \in \{1, \dots, N+1\}$, we will write $\tilde{V} = V \circ \Phi_\delta^{-1}$, and we will omit the dependence on $\delta$.

We introduce now some further notation. We will use the subscript $r$ to denote rescaled differential operators. More precisely,
\[
\nabla_r := 
\begin{pmatrix}
\nabla_x \\
\delta^{-1} \nabla_y
\end{pmatrix}
\quad
D_r := 
\begin{pmatrix}
\partial_x & \delta^{-1}\partial_y \\
\delta^{-1} \partial_x & \delta^{-2} \partial_y
\end{pmatrix}
\quad
\eps_r := \frac{1}{2}\bigg(D_r + (D_r)^T \bigg), 
\]
and therefore 
\[
\fa_r(\tilde{\beta}, \tilde{\eta}) = \frac{E}{12(1-\sigma^2)}\int_{\Omega_1} (1- \sigma) \eps_r(\tilde{\beta}) : \eps_r(\tilde{\eta}) + \sigma \Div_r \tilde{\beta} \Div_r \tilde{\eta}.
\]

With this notation, we can rescale \eqref{eq:weakPoi} and obtain the following weak problem. For all $\cJ^T\tilde{\eta} \in H^1(\Omega_1)^{N+1}$, $\tilde{v} \in H^1(\Omega)$, 
\begin{equation} \label{eq:weakPoiresc}
\begin{split}
&\frac{E}{12(1-\sigma^2)} \int_{\Omega_1} \bigg[ (1-\sigma) \eps_r(\tilde{\beta}_\delta) : \eps_r(\tilde{\eta}) + \sigma \Div_r \tilde{\beta}_\delta \Div_r \tilde{\eta} \bigg] dx dy \\
&+ \int_{\Omega_1} \bigg[\frac{Ek}{2(1-\sigma^2)t^2} (\nabla_r \tilde{w}_\delta - \cJ^T\tilde{\beta}_\delta) (\nabla_r \tilde{v} - \cJ^T\tilde{\eta}) + \frac{t^2}{12} \cJ^T\tilde{\beta}_\delta \cdot \cJ^T\tilde{\eta} + \tilde{w} \tilde{v} \bigg] dx dy \\
&= \int_{\Omega_1} \bigg[ \frac{t^2}{12} \cJ^T \tilde{F} \cdot \cJ^T \tilde{\eta} + \tilde{f} \tilde{v} \bigg] dx dy,
\end{split}
\end{equation}
which can be rewritten in the following form 
\begin{multline}
\label{Omega_1 problem weak form}
\fa_r(\tilde{\beta}, \eta) + \frac{\mu_1k}{t^2} (\nabla_r \tilde{w} - \cJ^T\tilde{\beta}, \nabla_r \tilde{v} - \cJ^T\eta) + \left[(\tilde{w},v) + \frac{t^2}{12}(\cJ^T\tilde{\beta}, \cJ^T\eta)\right] \\
= (\tilde{f},v) + \frac{t^2}{12}(\cJ^T\tilde{F}, \cJ^T\eta).
\end{multline}
for all $\eta \in (\cJ^{-T}H^1(\Omega_1)^{N})$, $v \in H^1(\Omega_1)$.

From \eqref{Omega_1 problem weak form} we deduce the apriori estimate
\begin{multline}
\frac{E}{12(1-\sigma^2)} \int_{\Omega_1} \Bigg((1-\sigma) \lvert\eps_r(\tilde{\beta})\rvert^2 + \sigma \lvert \Div_r\tilde{\beta}\rvert^2 + \frac{6k}{t^2} \lvert\nabla_r \tilde{w} - \cJ^T\tilde{\beta}\rvert^2 \Bigg)\\
+ |\tilde{w}|^2 + \frac{t^2}{12} |\cJ^T\tilde{\beta}|^2   \, dxdy 
\leq C \left( \norma{\tilde{f}}^2_{L^2(\Omega_1)} + \frac{t^2}{12} \norma{\cJ^T\tilde{F}}^2_{L^2(\Omega_1)} \right).
\end{multline}
Writing $\tilde{\beta} \in H^1(\Omega_1)^{N}$ as $\tilde{\beta} = (\tilde{\beta}^I, \tilde{\beta}^{II}) \in H^1(\Omega_1)^{N-d} \times H^1(\Omega_1)^{d}$, we have
\begin{equation} \label{apriori}
\begin{split}
&(1- \sigma) \Bigg( \bigg\lVert \eps_x (\tilde{\beta^I}) \bigg\rVert^2 + \bigg\lVert \frac{1}{\delta^2} \eps_y (\tilde{\beta}^{II}) \bigg\rVert^2 + \bigg\lVert \frac{1}{2 \delta} \bigg( \frac{\partial \tilde{\beta}^{I,i}}{\partial y_j} + \frac{\p \tilde{\beta}^{II,j}}{\partial x_i} \bigg) \bigg\rVert^2    \Bigg) \\
&+ \sigma \Bigg( \bigg\lVert \Div_x \tilde{\beta}^I + \frac{1}{\delta^2} \Div_y \tilde{\beta}^{II} \bigg\rVert^2 \Bigg) + \frac{1}{2} \big( \norma{\tilde{\beta^I}^2} + \bigg\lVert \frac{\tilde{\beta^{II}}}{\delta} \bigg\rVert + \norma{\tilde{w}}^2 \big) \\
&+ \bigg\lVert \nabla_x\tilde{w} - \tilde{\beta}^I \bigg\rVert^2 + \bigg\lVert \frac{1}{\delta} (\nabla_y \tilde{w} - \tilde{\beta}^{II} )  \bigg\rVert^2  \leq C\Big( \frac{t^2}{12} \lVert \cJ^T\tilde{F} \rVert^2 + \lVert \tilde{f} \rVert^2 \Big),
\end{split}
\end{equation}
where $\sup_{\delta > 0} \Big(\frac{t^2}{12} \lVert \cJ^T\tilde{F} \rVert^2 + \norma{\tilde{f}}^2\Big) \leq M$. Therefore we deduce that the sequences
\[
\begin{split}
&(\tilde{\beta^I})_{\delta > 0}, \quad\quad\quad\quad\quad\quad \bigg(\frac{\tilde{\beta}^{II}}{\delta}\bigg)_{\!\!\delta > 0}, \quad\quad\quad\quad\quad\quad (\tilde{w})_{\delta > 0}, \quad\quad (\eps_x( \tilde{\beta^I}))_{\delta > 0}, \\
& \bigg( \frac{1}{\delta^2} \eps_y(\tilde{\beta}^{II}) \bigg)_{\!\!\delta > 0},  \quad \bigg( \frac{1}{\delta} \bigg(\frac{\partial \tilde{\beta}^{I,i}}{\partial y_j} + \frac{\p \tilde{\beta}^{II,j}}{\partial x_i}\bigg)\bigg)_{\!\!\delta > 0}, \quad \big(\nabla_x \tilde{w}\big)_{\!\delta>0}, \quad \bigg( \frac{1}{\delta} \nabla_y \tilde{w} \bigg)_{\!\!\delta > 0},
\end{split}
\]
are $L^2$-bounded for $\delta > 0$. Therefore, up to a subsequence, the following limiting relations hold
\begin{align}
&\tilde{\beta}^{II} \rightharpoonup 0, \quad \frac{\tilde{\beta}^{II}}{\delta} \rightharpoonup H, \label{limit0}\\
&\frac{1}{\delta} \bigg(\frac{\partial \tilde{\beta}^{II,k}}{\partial y_j} + \frac{\partial \tilde{\beta}^{II,j}}{\partial y_k} \bigg)  \rightharpoonup 0, \quad \frac{1}{2\delta^2} \bigg(\frac{\partial \tilde{\beta}^{II,k}}{\partial y_j} + \frac{\partial \tilde{\beta}^{II,j}}{\partial y_k} \bigg) \rightharpoonup q_{jk}, \label{limit1}\\
&\tilde{\beta}^I \rightharpoonup \beta^I_0, \quad \frac{1}{2} \bigg( \frac{\partial \tilde{\beta}^{I,m}}{\partial x_i} +  \frac{\partial \tilde{\beta}^{I,i}}{\partial x_m} \bigg) \rightharpoonup  (\eps_x(\beta^I_0))_{mi}, \label{limit2}\\
&\frac{1}{2\delta} \bigg( \frac{\partial \tilde{\beta}^{I,i}}{\partial y_j} + \frac{\partial \tilde{\beta}^{II,j}}{\partial x_i}  \bigg) \rightharpoonup \Phi_{ij}, \label{limit3}\\
&\tilde{w} \overset{H^1}{\rightharpoonup} w_0, \quad \frac{1}{\delta} \frac{\partial \tilde{w}}{\p y_j} \rightharpoonup \omega_j, \label{limit4}
\end{align}
where $m,i \in \{1, \dots, N-d\}$, $k,j \in \{1, \dots, d\}$, and all the limits are $L^2$-functions defined in $\Omega_1$. We further assume that, up to a subsequence, there exist $\tilde{F}_0 \in L^2(\Omega_1)^N$ and $\tilde{f}_0 \in L^2(\Omega_1)$ such that 
\begin{equation} \label{eq:limit_data}
\cJ^T \tilde{F} \rightharpoonup \tilde{F}_0, \qquad \tilde{f} \rightharpoonup \tilde{f}_0,
\end{equation}
in $L^2(\Omega_1)$, as $\delta \to 0^+$. Note immediately that since $\cM_1 \cJ^T \tilde{F} = 0$, it must be $\tilde{F}_0^{II} = 0$. 

To characterise the limiting functions $\beta_0, H, q_{jk}, \Phi_{ij}$, \emph{etc.,} we now proceed to choose suitable test functions in \eqref{Omega_1 problem weak form}.\\[0.1cm]

\textbf{Step 1.} Choose $\eta = (\eta^I, \eta^{II})$, with $\eta^I = 0$ and $\eta^{II} =  \delta^2 (y_j - \cM_1(y_j)) \psi(x) {\mathbf e}_j$, with ${\mathbf e}_j$ the $j^{\rm th}$ vector of the canonical basis, in \eqref{Omega_1 problem weak form} for some $\psi \in H^1(\Omega)$, $v = 0$. Then $\eta \in H^1(\Omega)^N$ and $\cM_\delta \eta^{II} = 0$, hence $\eta$ is an admissible test function. We deduce that
\begin{multline}
\label{step1:limit}
\frac{E}{12(1-\sigma^2)}\int_{\Omega_1}\frac{(1-\sigma)}{\delta^2} \frac{\p \tilde{\beta}^{II,j}}{\p y_j} \psi + \frac{\sigma}{\delta^2} \Div_y \tilde{\beta}^{II} \psi + \sigma \Div_x \tilde{\beta}^I\psi  + o(1) \\
= \int_{\Omega_1} \frac{\tilde{F}^{II}}{\delta} \delta^2 (y_j - \cM_1(y_j)) \psi.
\end{multline}
Passing to the limit in \eqref{step1:limit}, taking into account \eqref{limit1}, \eqref{limit2}, gives
\[
\int_{\Omega} \bigg(\, (1-\sigma)q_{jj} + \sigma \sum_{i =1}^d q_{ii} + \sigma \Div_x \beta_0^I\, \bigg)\, \psi \, g \, dx = 0, \quad \textup{for all $\psi \in H_0^1(\Omega)$,}
\]
therefore $(1-\sigma) q_{jj} + \sigma \sum_{i =1}^d q_{ii} = - \sigma \Div_x \beta_0^I$, for all $j = 1, \dots, d$. Summing over $j$ yields
\[
\sum_{i=1}^d q_{ii} =  - \frac{d \, \sigma \Div_x \beta_0^I}{(1 - \sigma) + d\, \sigma},
\]
and therefore 
\[
q_{jj} = - \frac{\sigma \Div_x \beta_0^I}{(1 - \sigma) + d\, \sigma}, \qquad j = 1, \dots, d.
\]
A similar computation with $\eta^{II} =  \delta^2 (y_j - \cM_1(y_j)) \psi(x) {\mathbf e}_k$, $j \neq k$, $j,k \in \{1, \dots, d\}$, yields $q_{jk} = 0$, $j \neq k$. \\[0.1cm]

\textbf{Step 2.} Choose $i \in (1, \dots, N-d)$ and $j \in (1, \dots, d)$. Let $\psi \in H^1(\Omega)$. Then set $\eta^i = \delta y_j \psi(x)$, while all the other components are set to zero. Set moreover $v = 0$ in \eqref{Omega_1 problem weak form}. Such a choice implies
\[
\int_{\Omega_1} \frac{1}{4 \delta} \bigg(\frac{\p \tilde{\beta}^{I,i}}{\p y_j} + \frac{\p \tilde{\beta}^{II,j}}{\p x_i} \bigg) \psi + o(1)= \int_{\Omega_1} \tilde{F^I} \delta y_j \psi \to 0,
\]
and therefore we conclude that $\frac{1}{\delta} \big(\frac{\p \tilde{\beta}^{I,i}}{\p y_j} + \frac{\p \tilde{\beta}^{II,j}}{\p x_i} \big) \rightharpoonup 2\Phi_{ij} = 0$. \\ Due to \eqref{limit1} we see that
\[
-\int_{\Omega_1} \frac{1}{2\delta} \frac{\p \tilde{\beta^{II,j}}}{\p x_i} \xi =\int_{\Omega_1} \frac{1}{2\delta} \tilde{\beta^{II,j}} \frac{\p \xi}{\p x_i} \to - \int_{\Omega_1} \frac{H^j}{2} \frac{\p \xi}{\p x_i}
\]
for all $\xi \in H^1_0(\Omega_1)$. The uniform boundedness principle applied to the sequence of operators $T^{ij}_\delta \in H_0^1(\Omega_1)^* $ defined as
$$
T^{ij}_\delta(\xi) = \int_{\Omega_1} \frac{1}{2\delta} \frac{\p \tilde{\beta}^{II,j}}{\p x_i} \xi
$$ 
implies that $\sup_{\delta >0} \norma{T^{ij}_\delta}_{\cL(H^{-1}(\Omega_1))} < \infty$, which in turn yields $\sup_{\delta >0} \big \lVert \frac{1}{2\delta} \frac{\p \tilde{\beta}^{II,j}}{\p x_i} \big \rVert_{H^{-1}(\Omega_1)} < \infty$. We conclude that 
\[
\bigg \lVert \frac{\p \tilde{\beta}^{II,j}}{\p x_i} \bigg \rVert_{H^{-1}(\Omega_1)} \to 0,
\]
and that up to a subsequence
\[
\frac{1}{2\delta} \frac{\p \tilde{\beta}^{II,j}}{\p x_i} \rightharpoonup -\frac{1}{2} \frac{\p H_j}{\p x_i}
\]
in $H^{-1}(\Omega_1)$. In particular, up to a subsequence
\[
\frac{1}{2 \delta} \frac{\p \tilde{\beta}^{I,i}}{\p y_j} = \frac{1}{2\delta} \bigg( \frac{\p \tilde{\beta}^{I,i}}{\p y_j} + \frac{\p \tilde{\beta}^{II,j}}{\p x_i}\bigg) - \frac{1}{2\delta} \frac{\p \tilde{\beta}^{II,j}}{\p x_i} \rightharpoonup \frac{1}{2} \frac{\p H^j}{\p x_i},
\]
and therefore 
\[
\bigg \lVert \frac{\p \tilde{\beta}^{I,i}}{\p y_j} \bigg \rVert_{H^{-1}(\Omega_1)} \to 0.
\]
It must then be $\beta^I_0(x,y) = \beta^I_0(x)$, for a.a.\ $(x, y) \in \Omega_1$.\\[0.1cm]
%

\textbf{Step 3.}  Choose $\eta = (\eta^I, \eta^{II})$, $\eta^I(x,y) = \Psi(x)$, for $(x,y) \in \Omega_1$, $\eta^{II} = 0$, $v = 0$ in \eqref{Omega_1 problem weak form} for $\Psi \in H^1(\Omega)^{N-d}$. Then
\[
\int_{\Omega_1} (1-\sigma) \eps_x(\tilde{\beta^I}) : \eps_x(\Psi) + \sigma \bigg(\!\!\Div_x \tilde{\beta^I} + \frac{1}{\delta^2} \Div_y \tilde{\beta}^{II}  \bigg) \Div_x \Psi + \tilde{\beta^I} \cdot \Psi + o(1) = \int_{\Omega_1} \tilde{F}^I \cdot \Psi .
\]
By taking the limit as $\delta \to 0^+$ we deduce that
\begin{equation} \label{eq:limit-weak1}
\int_{\Omega_1} (1-\sigma) \eps_x(\beta_0^I) : \eps_x(\Psi) + \frac{(1-\sigma)\sigma}{(1-\sigma) + d \sigma} \Div_x \beta_0^I \Div_x \Psi + \beta_0^I \cdot \Psi = \int_{\Omega_1} \tilde{F}^I_{0} \cdot \Psi.
\end{equation}
Since all the functions except $\tilde{F}^I_{0}$ do not depend on $y$, \eqref{eq:limit-weak1} can be rewritten as
\[
\int_\Omega \bigg[(1-\sigma) \eps_x(\beta_0^I) : \eps_x(\Psi) + \frac{(1-\sigma)\sigma}{(1-\sigma) + d \, \sigma} \Div_x \beta_0^I \Div_x \Psi + \beta_0^I \cdot \Psi\bigg] g
= \int_\Omega \cM_1(\tilde{F}^I_{0}) \cdot \Psi g.
\]
for all $\Psi \in H^1(\Omega)^{N-d}$.
\vspace{0.2cm}

\textbf{Step 4.} Recall that $(\tilde{w}_\delta)_\delta$ is a uniformly bounded sequence in $H^1(\Omega_1)$ and, up to a subsequence, it converges weakly to $w_0 \in H^1(\Omega_1)$. Since
\[
\int_{\Omega_1} \bigg|\frac{\partial \tilde{w}}{\partial y_j}\bigg|^2 = O(\delta^2), \quad \textup{as $\delta \to 0^+$, \:\: $j=1, \dots, d$},
\]
we deduce that $w_0$ is constant in $y$. Now we are ready to pass to the limit as $\delta \to 0^+$ in \eqref{Omega_1 problem weak form}. Choose $\eta = (\eta^I, \eta^{II})$, $\eta^I(x,y) = \Psi(x)$,  with $\Psi \in C^{\infty}_c(\Omega)^{N-d}$, $\eta^{II} = 0$, $v(x, y) = v(x)$. Then
\begin{multline} \label{eq:limiteqOmega1}
\frac{E}{12(1- \sigma^2)}\int_{\Omega_1} (1-\sigma) \eps_x(\tilde{\beta}^I) : \eps_x(\Psi) + \sigma \big( \Div_x \tilde{\beta^I} + \frac{1}{\delta^2} \Div_y \tilde{\beta}^{II} \big) \Div_x \Psi \\
+ \frac{E k}{2(1-\sigma)t^2} \int_{\Omega_1} \big( \nabla_x \tilde{w}- \tilde{\beta^I} \big) \big(\nabla_x v - \Psi \big) +  \int_{\Omega_1} \tilde{w} v + \frac{t^2}{12} \tilde{\beta^I} \cdot \Psi + o(1) = \int_{\Omega_1} \tilde{F}^I \cdot \Psi + \tilde{f} v.
\end{multline}
Due to Steps 1-3, the limit as $\delta \to 0^+$ of \eqref{eq:limiteqOmega1} is
\begin{multline} \label{eq:limiteqOmega12}
\frac{E}{12(1- \sigma^2)}\int_{\Omega_1}(1-\sigma) \eps_x(\beta_0^I) : \eps_x(\Psi) + \frac{(1-\sigma)\sigma}{(1-\sigma) + d \sigma} \Div_x \beta_0^I \Div_x \Psi \\
+ \frac{E k}{2(1-\sigma)t^2} \int_{\Omega_1} \big( \nabla_x \tilde{w} - \beta_0^I \big) \big(\nabla_x v - \Psi \big) +  \int_{\Omega_1} \tilde{w} v + \frac{t^2}{12} \beta_0^I \Psi + o(1) = \int_{\Omega_1} \tilde{F}^I \Psi + \tilde{f} v.
\end{multline}
Note that all the functions appearing in \eqref{eq:limiteqOmega12} do not depend on $y$, with the sole possible exceptions of $\tilde{F}^I$ and $\tilde{f}$. Nevertheless, we have the elementary equality
\[
\int_\Omega \int_{\Omega_1^{II}(x)} \tilde{F}^I \Psi + \tilde{f} v dy dx = \int_\Omega g(x) (\cM_1(\tilde{F}^I) \Psi + \cM_1(\tilde{f}) v) dx.
\]
As a consequence, taking the limit in \eqref{eq:limiteqOmega12} yields
\begin{multline} \label{limitqform}
\frac{E}{12(1- \sigma^2)}\int_{\Omega} \bigg((1-\sigma) \eps_x(\beta_0^I) : \eps_x(\Psi) + \frac{(1-\sigma)\sigma}{(1-\sigma) + d \sigma} \Div_x \beta_0^I \Div_x \Psi\bigg) g \\
+ \frac{E k}{2(1-\sigma)t^2} \int_{\Omega} \big( \nabla_x w_0 - \beta_0^I \big) \big(\nabla_x v - \Psi \big) g
+\int_{\Omega} \big(w_0 v + \frac{t^2}{12} \beta_0^I \Psi + o(1) \big) g \\
= \int_{\Omega} (\cM_1\tilde{F}_0^I \Psi + \cM_1\tilde{f}_0 v) g.
\end{multline}
Note that problem \eqref{limitqform} admits the following strong formulation
\[
\begin{cases}
-\frac{E}{24(1+ \sigma)} \Div(g \nabla(\beta_0^I)) -  \frac{E(1 + (d+1) \sigma )}{24 (1 + \sigma)(1 + (d-1) \sigma)}  \nabla( g \Div\beta_0^I) &  \\
\hspace{2.5cm} + \frac{E k}{2(1-\sigma)t^2} (\nabla w_0 - \beta_0^I) g + \frac{t^2}{12} \beta_0^I g  = \cM_1(\tilde{F}_0^I) g , \quad &\textup{in $\Omega$,} \\
-\frac{E k}{2(1-\sigma)t^2} \Div((\nabla w_0 - \beta_0^I) g ) + w_0 g = \cM_1(\tilde{f}) g, \quad &\textup{in $\Omega$,} \\
\frac{\mu_1}{6} (\nu^T \eps(\beta^I_0) \nu) + \frac{\mu_2}{12} (\Div \beta^I_0) = 0, \quad &\textup{on $\partial \Omega$,}\\
(\eps(\beta^I_0) \nu)_{\p \Omega}= 0, \quad   &\textup{on $\partial \Omega$,} \\
(\nabla w_0 - \beta^I_0) \cdot \nu = 0, \quad &\textup{on $\p \Omega$.}
\end{cases}
\]

\vspace{0.2cm}\textbf{Conclusion.} Recalling the definition of $\cE$-convergence with respect to the connecting system $((\Omega_\delta)_\delta, \Omega_0, (\cE_\delta)_\delta)$, we first assume that $G_\delta := (F^I_\delta, \delta F^{II}_\delta, f_\delta) \in \cH_\delta$ and
\begin{equation} \label{E_delta_conv}
\norma{G_\delta - \cE_\delta G_0}_{\cH_\delta} \to 0
\end{equation}
for some $G_0 = (F_0, f_0) \in \cH_0$. Note that
\[
\begin{split}
\norma{G_\delta - \cE_\delta G_0}^2_{\cH_\delta} &= \int_{\Omega_\delta} \big(|F_\delta^I - \cE_\delta F_0|^2 + \delta^2|F_\delta^{II}|^2 + |f_\delta - \cE_\delta f_0|^2 \big) \delta^{-d} dx \\
&= \int_{\Omega_1} |\tilde{F}_\delta^I - \cE F_0|^2 + |\tilde{F}_\delta^{II}|^2 + |\tilde{f}_\delta - \cE f_0|^2 dx \\
&= \norma{\cJ^T\tilde{G}_\delta - \cE G_0}^2_{L^2(\Omega_1)^{N+1}},
\end{split}
\]
hence, \eqref{E_delta_conv} is equivalent to $\norma{\cJ^T\tilde{G}_\delta - \cE G_0}_{L^2(\Omega_1)^{N+1}} \to 0$ as $\delta \to 0^+$. We consider the problem
\begin{equation}
\label{bvp-Poi-final}
\begin{cases}
- \frac{\mu_1}{12} \Delta \beta - \frac{\mu_1 + \mu_2}{12} \nabla(\Div \beta) - \frac{\mu_1 k}{t^2} (\nabla w - \beta) + \beta = \frac{t^2}{12} F_\delta, &\textup{in $\Omega_\delta$,}\\
- \frac{\mu_1 k}{t^2} ( \Delta w - \Div \beta) + w = f_\delta, &\textup{in $\Omega_\delta$,}\\
\frac{\mu_1}{6} (\nu^T \eps(\beta) \nu) + \frac{\mu_2}{12} (\Div \beta) = 0, &\textup{on $\partial \Omega_\delta$,}\\
(\eps(\beta) \nu)_{\p \Omega_\delta}= 0,   &\textup{on $\partial \Omega_\delta$,}\\
(\nabla w - \beta) \cdot \nu = 0, &\textup{on $\partial \Omega_\delta$.}
\end{cases}
\end{equation}
Due to step 1-4 and to \eqref{E_delta_conv}, we know that the rescaled solution $(\tilde{\beta}_\delta, \tilde{w}_\delta)$ of \eqref{bvp-Poi-final} verifies
\[
\tilde{\beta}^I \rightharpoonup \beta^I_0, \quad \tilde{w} \rightharpoonup w_0,
\]
in $L^2(\Omega_1)$, where $(\beta^I_0, w_0)$ solves
\begin{equation}\label{finallimit}
\begin{cases}
-\frac{E}{24(1+ \sigma)} \Div(g \nabla(\beta_0^I)) -\frac{E(1 + (d+1) \sigma)}{24 (1 + \sigma)(1 + (d-1) \sigma)} \nabla( g \Div\beta_0^I) & \\
\hspace{2.8cm} + \frac{E k}{2(1-\sigma)t^2} (\nabla w_0 - \beta_0^I) g + \frac{t^2}{12} \beta_0^I g  = \tilde{F}_0 g , &\textup{in $\Omega$,} \\
-\frac{E k}{2(1-\sigma)t^2} \Div((\nabla w_0 - \beta_0^I) g  ) + w_0 g  = \tilde{f}_0 g , \quad &\textup{in $\Omega$,} \\
\frac{E}{12(1-\sigma)} (\nu^T \eps(\beta_0^I) \nu_\Omega) + \frac{\sigma E}{24(1-\sigma^2)} (\Div \beta_0^I) = 0, &\textup{on $\partial \Omega$,}\\
(\eps(\beta^I_0) \nu_\Omega)_{\p \Omega}= 0,   &\textup{on $\partial \Omega$,}\\
(\nabla w_0 - \beta^I_0) \cdot \nu_\Omega = 0, &\textup{on $\partial \Omega$.}
\end{cases}
\end{equation}
Moreover
\[
\eps(\tilde{\beta}^I) \rightharpoonup \eps(\beta^I_0), \quad \Div \tilde{\beta}^I \rightharpoonup \Div \beta^I_0, 
\]
and by Korn's inequality in the fixed domain $\Omega_1$ we conclude that
\[
\norma{\nabla (\tilde{\beta}^I - \beta^I_0)}_{L^2(\Omega_1)} \leq C_{\Omega_1} \big(\norma{\eps(\tilde{\beta}^I - \beta^I_0)}_{L^2(\Omega_1)} + \norma{\Div(\tilde{\beta}^I - \beta^I_0) }_{L^2(\Omega_1)} + \norma{\tilde{\beta}^I_\delta - \beta^I_0}_{L^2(\Omega_1)}\big).
\]
Hence $(\tilde{\beta}^I - \beta^I_0)$ is a uniformly bounded sequence in $H^1(\Omega_1)^{N-d}$; up to a subsequence, there exists $\gamma \in H^1(\Omega_1)^{N-d}$ such that $(\tilde{\beta}^I - \beta^I_0) \rightharpoonup \gamma$ in $H^1(\Omega_1)^{N-d}$, but since $(\tilde{\beta}^I - \beta^I_0) \rightharpoonup 0$ in $L^2(\Omega_1)^{N-d}$, it must be $\gamma = 0$, hence 
\[
(\tilde{\beta}^I - \beta^I_0) \rightharpoonup 0, \quad \textup{in $H^1(\Omega_1)^{N-d}$.}
\]
Now, by the Rellich-Kondrachov Theorem,  $H^1(\Omega_1)^{N-d}$ is compactly embedded in $L^2(\Omega_1)^{N-d}$; thus,
\[
\tilde{\beta}^I \to \beta^I_0, \quad \textup{in $L^2(\Omega_1)^{N-d}$.}
\]
Again by the Rellich-Kondrachov Theorem, $\tilde{w} \to w_0$ in $L^2(\Omega_1)$. 
Regarding $\delta^{-1} \tilde{\beta}_\delta^{II}$, since $\cM_1(\tilde{\beta}_\delta^{II}) = 0$, we deduce immediately that
\[
\begin{split}
\frac{\norma{\tilde{\beta}_\delta^{II}}^2}{\delta^2} = \int_{\Omega_1} \bigg|\frac{\tilde{\beta}_\delta^{II} - \cM_1(\tilde{\beta}_\delta^{II})}{\delta}\bigg|^2 &\leq C_{\rm Poi} \int_{\Omega_1} \bigg|\frac{\nabla_y \tilde{\beta}_\delta^{II}}{\delta}   \bigg|^2  = C_{\rm Poi} \delta^2 \int_{\Omega_1} \bigg|\frac{\nabla_y \tilde{\beta}_\delta^{II}}{\delta^2}   \bigg|^2 \\
&\leq C_{\rm Poi} C_{\rm Korn} \delta^2 \bigg( \int_{\Omega_1} \frac{|\eps_y (\tilde{\beta}_\delta^{II})|}{\delta^2} + \int_{\Omega_1} \frac{|\tilde{\beta}^{II}_\delta|^2}{\delta^2} \bigg),
\end{split}
\]
hence
\[
\frac{\norma{\tilde{\beta}_\delta^{II}}^2}{\delta^2} \leq \frac{C_{\rm Poi} C_{\rm Korn}}{1- C_{\rm Poi} C_{\rm Korn} \delta^2} \delta^2 \big(\norma{\tilde{f}}_{L^2(\Omega_1)} + \norma{\cJ^T \tilde{F}}^2_{L^2(\Omega_1)^N}\big) \to 0.
\]

Therefore we conclude that
\[
\norma{(\beta_\delta, w_\delta) - \cE_\delta(\beta^I_0, w_0)}_{\cH_\delta} = \norma{(\tilde{\beta}^I_\delta, \tilde{w}_\delta) - \cE(\beta^I_0, w_0)}_{L^2(\Omega_1)^{N-d+1}}  + \delta^{-2} \norma{\tilde{\beta}_\delta^{II}}^2_{L^2(\Omega_1)^d}\to 0.
\]
This shows $(iii)(a)$ in Definition \ref{def:Econv}. To establish $(iii)(b)$, assume instead that $G_\delta := (F_\delta, f_\delta) \in \cH_\delta$ and $\sup_{\delta > 0} \norma{G_\delta}_{\cH_\delta} < \infty$, or equivalently, $\sup_{\delta > 0} \norma{\cJ^T\tilde{G}_\delta}_{L^2(\Omega_1)^{N+1}} < + \infty$. Then, up to a subsequence, $\cJ^T\tilde{G}_\delta = (\tilde{F}^I_\delta, \tilde{F}^{II}_\delta, \tilde{f}_\delta)$ converges weakly in $L^1(\Omega_1)^{N+1}$ to $(\tilde{F}^I_{0}, 0, \tilde{f}_0)$, with $\tilde{F}$ and $\tilde{f}$ as in \eqref{eq:limit_data}. Here we implicitly used that $\tilde{F}^{II}_0 = 0$ because it is the weak limit of vector fields with null integral average. Then from Step 1-4, we deduce that up to a subsequence,
the solution $(\cJ^T\tilde{\beta}_\delta, \tilde{w}_\delta)$, which is obtained from the solution of \eqref{bvp-Poi-final} after rescaling, converges weakly in $L^2(\Omega_1)^{N+1}$ to the solution $(\beta^I_0, w_0)$ of problem \eqref{finallimit}, with data $(\cM_1 \tilde{F}^I_0, \cM_1 \tilde{f}_0)$. An application of Korn's inequality and the Rellich-Kondrachov Theorem as done before yields that the convergence is in fact strong in $L^2(\Omega_1)^{N+1}$, therefore concluding the proof. 
\end{proof}

\begin{rem} \label{rmk:quantitativeconv}
Let $y_j^0 := y_j - M_1(y_j)$. The proof of Theorem \ref{main} gives some quantitative estimates of the convergence. Recall that in Step 1 we proved that for all $\psi \in H^1_0(W)$,
\begin{multline} \label{eq:remainder}
\int_{\Omega_1}\frac{(1-\sigma)}{\delta^2} \frac{\p \tilde{\beta}^{II,j}}{\p y_j} \psi + \frac{\sigma}{\delta^2} \Div_y \tilde{\beta}^{II} \psi + \sigma \Div_x \tilde{\beta}^I\psi \\
- \bigg(\int_{\Omega_1} \Big((1-\sigma)q_{jj} + \sigma \sum_{i =1}^d q_{ii} + \sigma \Div_x \beta_0^I\Big)\, \psi\, g\, \bigg) = \int_{\Omega_1} \frac{\tilde{F}^{II,j}}{\delta} \delta^2 y^0_j \psi + R^\delta_j \to 0,
\end{multline}
where 
\[
\begin{split}
R^\delta_j(\beta, \psi) &= - \frac{E}{24(1-\sigma)^2}\int_{\Omega_1} (1-\sigma) \frac{1}{2 \delta} \bigg( \frac{\p \tilde{\beta}^{I,i}}{\p y_j} + \frac{\p \tilde{\beta}^{II,j}}{\p x_i} \bigg) \frac{\delta}{2} y_j^0 \frac{\p \psi}{\p x_i} \\
&- \frac{Ek}{2(1-\sigma)t^2}\int_{\Omega_1} \frac{1}{\delta}(\p_{y_j} \tilde{w} - \tilde{\beta}^{II,j}) \delta y_j^0 \psi - \int_{\Omega_1} \frac{\tilde{\beta}^{II,j}}{\delta} \delta y_j \psi + \int_{\Omega_1} \frac{\tilde{F}^{II,j}}{\delta} \delta y_j^0 \psi \\
&= - \frac{E}{24(1-\sigma)^2}\int_{\Omega_1} (1-\sigma) \frac{1}{2 \delta} \bigg( \frac{\p \tilde{\beta}^{I,i}}{\p y_j} + \frac{\p \tilde{\beta}^{II,j}}{\p x_i} \bigg) \frac{\delta}{2} y_j^0 \frac{\p \psi}{\p x_i} \\
&\hspace{0.37cm}- \frac{Ek}{2(1-\sigma)t^2}\int_{\Omega_1} \frac{1}{\delta}(\p_{y_j} \tilde{w}) \delta y_j^0 \psi
\end{split}
\]
and
\[
q_{jj} = - \frac{\sigma \Div_x \beta_0^I}{(1 - \sigma) + d\, \sigma}, \qquad j = 1, \dots, d.
\]
Upon summing over $j \in \{1, \dots, d\}$ in \eqref{eq:remainder} we get
\begin{multline}\label{eq:L}
\int_{\Omega_1}\frac{(1-\sigma) + d\sigma}{\delta^2} \Div_y \tilde{\beta}^{II} \psi + d \sigma \Div_x \tilde{\beta}^I\psi \\
- \bigg(\int_{\Omega_1} \Big((1-\sigma + d \sigma) \sum_{i =j}^d q_{jj} + d \sigma \Div_x \beta_0^I \Big) \, \psi \, g\bigg) = \int_{\Omega_1} \frac{\tilde{F}^{II} \cdot y^0}{\delta} \delta^2 \psi + \tilde{R}^\delta_j\to 0,
\end{multline}
where $\tilde{R}^\delta_j = \sum_j R^\delta_j$. Note that $(1-\sigma + d \sigma) \sum_{i =j}^d q_{jj} + d \sigma \Div_x \beta_0^I = 0$, so
\[
\int_{\Omega_1}\frac{1}{\delta^2} \Div_y \tilde{\beta}^{II} \psi + \frac{d \sigma}{1 - \sigma + d\sigma} \Div_x \tilde{\beta}^I\psi = \frac{1}{1-\sigma + d\sigma} \bigg(\int_{\Omega_1} \frac{\tilde{F}^{II} \cdot y^0}{\delta} \delta^2 \psi + \tilde{R}^\delta_j \bigg)\to 0.
\]
We claim that there exists $C_L > 0$ such that
\begin{equation} \label{eq:R_delta}
 |R_j^\delta(\beta, \psi)| \leq C_L\, \delta \, \Big( \frac{t^2}{12} \lVert \cJ^T\tilde{F} \rVert^2 + \lVert \tilde{f} \rVert^2 \Big) \norma{\psi}_{ H^1(\Omega)}. 
\end{equation}
For instance, we have
\begin{multline} \label{eq:boh}
\bigg \lvert \frac{E}{24(1-\sigma)^2}\int_{\Omega_1} (1-\sigma) \frac{1}{2 \delta} \bigg( \frac{\p \tilde{\beta}^{I,i}}{\p y_j} + \frac{\p \tilde{\beta}^{II,j}}{\p x_i} \bigg) \frac{\delta}{2} y_j^0 \frac{\p \psi}{\p x_i} \bigg \rvert \\
\leq C_L \bigg \lVert \frac{1}{2 \delta} \bigg( \frac{\p \tilde{\beta}^{I,i}}{\p y_j} + \frac{\p \tilde{\beta}^{II,j}}{\p x_i} \bigg) \bigg \rVert \delta \norma{\p_{x_i} \psi}.
\end{multline}
Moreover, from \eqref{apriori} we have
\[
\bigg \lVert \frac{1}{2 \delta} \bigg( \frac{\p \tilde{\beta}^{I,i}}{\p y_j} + \frac{\p \tilde{\beta}^{II,j}}{\p x_i} \bigg) \bigg \rVert \leq C_L \Big( \frac{t^2}{12} \lVert \cJ^T\tilde{F} \rVert^2 + \lVert \tilde{f} \rVert^2 \Big).
\]
Therefore we conclude that the right-hand side of \eqref{eq:boh} can be estimated from above by $C_L \delta \norma{\psi}_{H^1(\Omega)} \Big( \frac{t^2}{12} \lVert \cJ^T\tilde{F} \rVert^2 + \lVert \tilde{f} \rVert^2 \Big)$. Arguing similarly for each integral appearing in $R(\delta)$ we finally deduce \eqref{eq:R_delta}.
\end{rem}

\section{Norm resolvent convergence in the sense of Vainikko} \label{Sec:normresconv}
In Theorem \ref{sp_conv} we proved that $A_\delta^{-1} \to A_0^{-1}$ in compact convergence sense, with respect to the connecting system $((\cE_\delta)_\delta, (\cH_\delta)_\delta, \cH_0)$. One could ask if this convergence happens uniformly with respect to Vainikko norm resolvent convergence, see Definition \ref{def:normresconv}.


\begin{theorem} \label{normresconv}
In the notation of Theorem \ref{sp_conv}, $A_\delta \to A_0$ in generalised norm resolvent convergence.  
\end{theorem}
\begin{proof}
Let us set $B_\delta = A_\delta^{-1}$ and $B_0 = A_0^{-1}$. First note that 
$$\norma{B_\delta \cE_\delta - \cE_\delta B_0}_{\cL(\cH_0, \cH_\delta)} = \norma{T_\delta}_{\cL(\cH_0, \cH_1)} := \norma{\tilde{B}_\delta \cE- \cE B_0}_{\cL(\cH_0, \cH_1)},$$
where $\tilde{B}_\delta$ is the rescaled operator appearing in the proof of Theorem \ref{main}, and for $F \in \cH_0$, $\cE F(x,y) = F(x)$ for all $(x,y) \in \Omega$. We observe that Theorem \ref{main} shows that $T_\delta F_0\to 0$ in $L^2(\Omega_1)^{N+1}$, for any $F_0\in L^2(\Omega)^{N-d+1}$, and moreover that $T_\delta\to 0$ compactly. In other words, $(T_\delta)_\delta$ is a collectively compact family of operators in the sense of Anselone and Palmer (see Definition \ref{def:collcomp}).  Our strategy will be to use Theorem \ref{cor:normconv} on the family of operators $(T_\delta)_\delta$. For this, we only need to check that $(T_\delta)_\delta$ satisfies \eqref{eq:continuityprop} with $T_0 = 0$.\\
Let then $(u_\delta)_\delta$ be a bounded sequence in $\cH_0$ and assume without loss of generality that $u_\delta \rightharpoonup u$ in $\cH_0$ as $\delta \to 0^+$. By definition of $T_\delta$ we have that 
\[
T_\delta u_\delta = (\tilde{B}_\delta \cE - \cE B_0) u_\delta.
\]
Since $B_0$ is compact, we immediately find that $B_0 u_\delta \to B_0 u$. Moreover, from the proof of Theorem \ref{sp_conv},
\begin{equation} \label{cont_property}
\lim_{\delta \to 0^+}\tilde{B}_\delta (\cE u_\delta) = B_0 \cE (\cM_1{\rm w-}\lim_{\delta \to 0^+} u_\delta) =  B_0 ({\rm w-}\lim_{\delta \to 0^+} u_\delta) = B_0 u,
\end{equation}
where the convergence is possibly up to a subsequence. Thus, possibly up to a subsequence,
\[
T_\delta u_\delta \to B_0 u - \cE B_0 u  = 0,
\]
where we have used that since $B_0 u \in \cH_0$, $\cE B_0 u = u$. \\
Thus, $(T_\delta)_\delta$ has the required property \eqref{eq:continuityprop}; by Theorem \ref{cor:normconv}, $\norma{T_\delta}_{\cL(\cH_0, \cH_1)} \to 0$ as $\delta \to 0^+$, concluding the proof.
\end{proof}

Theorem \ref{normresconv} establishes that $A^{-1}_\delta \cE_\delta - \cE_\delta A^{-1}_0$ converges in $L(\cH_0, \cH_\delta)$ to zero, so in particular there exists a function $\omega(\delta)$, $\omega(\delta) \to 0$ as $\delta \to 0^+$, such that
\begin{equation} \label{rateconv}
\norma{A^{-1}_\delta \cE_\delta f_0 - \cE_\delta A^{-1}_0 f_0}_{\cH_\delta} \leq \omega(\delta) \norma{f_0}_{\cH_0}
\end{equation}
for all $f_0 \in \cH_0$. Note that the proof of Theorem \ref{normresconv} remains valid in a more general setting in which we have a family $(A_\delta)_\delta$ of non-negative self-adjoint operators in $\cH_\delta$ converging compactly in generalised sense, with the additional property that \eqref{cont_property} holds. The abstract nature of the proof does not give information on the rate of convergence $\omega(\delta)$ as $\delta \to 0^+$. 

Obtaining the sharp rate of convergence is not an easy task. We note that this is already not trivial for the Neumann Laplacian on curved thin tubes; and it is yet open for higher order elliptic operators, such as the biharmonic operator with free boundary conditions, where it is not possible to separate variables. \\
For the Reissner-Mindlin system, the main obstruction to obtain a sharp rate of convergence is given by the lack of a uniform estimate for $\nabla_y \beta_\delta^{I}$; this is a consequence of the lack of a uniform Korn inequality for the family of domains $\Omega_\delta$. Despite all these issues we state the following\\[0.2cm]
\textbf{Conjecture.} The rate of convergence $\omega(\delta)$ in \eqref{rateconv} is $\delta^{1/2}$.\\[0.2cm]
To support this conjecture, we show that $\omega(\delta) \leq C\delta^{1/2}$ for $\Omega_\delta = \Omega \times B_d(0, \delta)$.
\begin{theorem} \label{thm:rateconv} Assume that $\Omega$ is of class $C^2$ and let $\Omega_\delta = \Omega \times B_d(0,\delta)$ or $\Omega_\delta = \Omega \times (-\delta/2, \delta/2)^d$. Let $f_0 \in \cH_0$, then
\[
\norma{B_\delta \cE_\delta f_0 - \cE_\delta B_0 f_0}_{\cH_\delta} \leq C \delta^{1/2} \norma{f_0}_{\cH_0}.
\]
\end{theorem}
\begin{proof} 
To prove this theorem we adopt a variational approach. We define the energy functional
\[
\begin{split}
\sF_\delta(\eta_\delta, v_\delta) &:= \int_{\Omega_\delta} \bigg(\frac{E}{24 (1-\sigma^2)}\big[(1-\sigma) |\eps(\eta_\delta)|^2 + \sigma |\Div \eta_\delta|^2\big] \\
&+ \frac{Ek}{4(1-\sigma)t^2} |\nabla v_\delta - \eta_\delta|^2 + \frac{1}{2} \bigg[ \frac{t^2}{12}|\eta_\delta|^2 + |v_\delta|^2 \bigg] - \cE_\delta f_0 \cdot (\eta_\delta, v_\delta)^T \bigg) \delta^{-d} dxdy, 
\end{split}
\]
for $(\eta_\delta, v_\delta) \in \dom (\sF_\delta)$, where $\dom(\sF_\delta)\subseteq\cH_\delta$ densely. We note that 
\begin{multline} \label{positivity}
\sF^{\rm hom}_\delta(\eta_\delta, v_\delta) :=  \sF_\delta(\eta_\delta, v_\delta) + \int_{\Omega_\delta} \cE_\delta f_0 \cdot (\eta_\delta, v_\delta)^T \delta^{-d} dxdy \\
\geq \min\bigg\{\frac{t^2}{24}, \frac{1}{2} \bigg\} \norma{(\eta_\delta, v_\delta)}^2_{\cH_\delta} \geq 0.
\end{multline}
We further define 
\[
\begin{split}
\sF_0(\eta_0, v_0) &:= \int_{\Omega} \bigg( \frac{E}{24 (1-\sigma^2)}\big[(1-\sigma) |\eps_x(\eta_0)|^2 + \frac{(1-\sigma)\sigma}{(1-\sigma) + d \sigma} |\Div_x \eta_0|^2\big] \\
&+ \frac{Ek}{4(1-\sigma)t^2} |\nabla v_0 - \eta_0|^2 + \frac{1}{2} \bigg[ \frac{t^2}{12}|\eta_0|^2 + |v_0|^2 \bigg] - f_0 \cdot (\eta_0, v_0)^T \bigg)g(x) dx,
\end{split}
\]
for $(\eta_0, v_0) \in \dom (\sF_0)$, where $\dom(\sF_0)\subseteq\cH_0$ densely. Finally, set
\[
\begin{split}
\la_\delta := \min_{(\eta, v) \in \dom(\sF_\delta)} \sF_\delta(\eta, v), \\
\mu := \min_{(\eta, v) \in \dom(\sF_0)} \sF_0(\eta, v),
\end{split}
\]
and let $(\beta_\delta, w_\delta), (\beta_0, w_0)$ be the (unique) solutions to
\[
\la_\delta = \sF_\delta(\beta_\delta, w_\delta), \qquad \mu = \sF_0(\beta_0, w_0).
\]
Note that due to the Euler-Lagrange variational principle, $\Phi_\delta := (\beta_\delta, w_\delta)$ is exactly the solution of the equation $A_\delta \Phi_\delta = \cE_\delta f_0$, and $\Phi_0 := (0, w_0)$ is the solution of $A_0 \Phi_0 = f_0$. \\
Then it is easy to see that
\begin{equation} \label{eq:upperbound}
\la_\delta \leq \sF_\delta(\cE_\delta \beta_0, \cE_\delta w_0) = \mu + I(\beta_0, \delta),
\end{equation}
where 
\[
 I(\beta_0, \delta) =\frac{E}{24(1-\sigma^2)} \int_{\Omega_\delta} \frac{\sigma^2\, d}{(1-\sigma) + \sigma\, d} |\Div_x \cE_\delta(\beta_0)|^2 dx \delta^{-d}dy.
\]
We claim that 
\begin{equation} \label{eq:lowerbound}
\begin{split}
\la_\delta &\geq \mu + I(\beta_0, \delta) + \sF_\delta(\beta_\delta - \cE_\delta \beta_0, w_\delta - \cE_\delta w_0) \\
&-\frac{E \sigma}{24 (1-\sigma^2)}\frac{\sigma \, d}{(1-\sigma) + \sigma d} \norma{\Div_x (\beta_\delta^I - \cE_\delta \beta_0)}^2_{\cH_\delta} \\
&+ \int_{\Omega_\delta}\cE_\delta f_0 \cdot (\beta_\delta- \cE_\delta \beta_0, w_\delta - \cE_\delta w_0)^T\delta^{-d} dxdy + 2 L(\beta_\delta, \beta_0, \delta)
\end{split}
\end{equation}
for some functional $L$ with the property that $|L(\beta_\delta, \beta_0, \delta)| \leq C_L \delta \norma{f_0}_{\cH_0} \norma{\Div \beta_0}_{H^1(\Omega)}$.\\
To achieve the lower bound \eqref{eq:lowerbound}, first write
\begin{equation} \label{eq:ladelta}
\begin{split}
\la_\delta = \sF_\delta(\beta_\delta, w_\delta) = &\sF_\delta(\beta_\delta - \cE_\delta \beta_0, w_\delta - \cE_\delta w_0) \\
&+ \sF_\delta(\cE_\delta \beta_0, \cE_\delta w_0) + 2 \mathcal R_\delta(\beta_\delta- \cE_\delta \beta_0, w_\delta - \cE_\delta w_0; \beta_0, w_0)
\end{split}
\end{equation}
with
\[
\begin{split}
&\mathcal R_\delta(\beta_\delta- \cE_\delta \beta_0, w_\delta - \cE_\delta w_0; \beta_0, w_0) := \\
&\int_{\Omega_\delta} \bigg(\frac{E}{24 (1-\sigma^2)}\big[(1-\sigma) \eps(\beta_\delta- \cE_\delta \beta_0):\eps(\beta_0) + \sigma \Div (\beta_\delta- \cE_\delta \beta_0) \Div\beta_0 \big] \\
&+ \frac{Ek}{4(1-\sigma)t^2} (\nabla (w_\delta - \cE_\delta w_0) - (\beta_\delta- \cE_\delta \beta_0))\cdot (\nabla w_0 - \beta_0) \\
&+ \frac{1}{2} \bigg[ \frac{t^2}{12}(\beta_\delta- \cE_\delta \beta_0)\beta_0 + ( w_\delta - \cE_\delta w_0)w_0 \bigg] \bigg) \delta^{-d} dxdy.
\end{split}
\]
Since the measure of the $x$-sections do not depend on $x$, the derivatives commute with the averaging operator $\cM_\delta$; we can then rewrite the previous formula as
\begin{equation} \label{id2}
\begin{split}
&{\mathcal R_\delta}(\beta_\delta- \cE_\delta \beta_0, w_\delta - \cE_\delta w_0; \beta_0, w_0) := \\
&= \int_{\Omega} \bigg( \frac{E}{24 (1-\sigma^2)}\big[(1-\sigma) \eps_x(\cM_\delta\beta_\delta- \beta_0):\eps_x(\beta_0) + \frac{(1-\sigma)\sigma}{(1-\sigma) + d \sigma} \Div_x (\cM_\delta\beta_\delta- \beta_0) \Div_x\beta_0\big] \\
&+ \frac{Ek}{4(1-\sigma)t^2} (\nabla_x (\cM_\delta w_\delta - w_0) - (\cM_\delta\beta_\delta- \beta_0))\cdot (\nabla_x w_0 - \beta_0)\\
&+ \frac{1}{2}\bigg[ \frac{t^2}{12}(\cM_\delta\beta_\delta- \beta_0)\beta_0 + (\cM_\delta w_\delta - \cE_\delta w_0)w_0  \bigg] \bigg)g(x) dx\\
&-\frac{E \sigma}{24 (1-\sigma^2)} \int_{\Omega_\delta}\frac{\sigma \, d}{(1-\sigma) + \sigma d} |\Div_x (\beta_\delta^I - \cE_\delta \beta_0)|^2 dx \delta^{-d} dy +  L(\beta_\delta, \beta_0, \delta),
\end{split}
\end{equation}
where 
\[
L(\beta_\delta, \beta_0, \delta) = \frac{E \sigma}{24 (1-\sigma^2)} \int_{\Omega_\delta} \bigg( \Div_y \beta_\delta^{II} + \frac{\sigma \, d}{(1-\sigma) + \sigma d} \Div_x \beta_\delta^I \bigg) \Div_x \beta_0 dx \delta^{-d} dy.
\]
Note that, since $(\beta_0,w_0)^T = B_0 f_0$, 
\begin{equation} \label{id3}
\begin{split}
&2 \mathcal R_\delta(\beta_\delta- \cE_\delta \beta_0, w_\delta - \cE_\delta w_0; \beta_0, w_0)\\
&=\int_{\Omega_\delta}\cE_\delta f_0 \cdot (\beta_\delta- \cE_\delta \beta_0, w_\delta - \cE_\delta w_0)^T \delta^{-d} dxdy \\
&-\frac{E \sigma}{24 (1-\sigma^2)}\frac{\sigma \, d}{(1-\sigma) + \sigma d} \norma{\Div_x (\beta_\delta^I - \cE_\delta \beta_0)}^2_{\cH_\delta} + 2 L(\beta_\delta, \beta_0, \delta).
\end{split}
\end{equation}

Moreover,
\begin{equation}
\label{id1}
\sF_\delta(\cE_\delta \beta_0, \cE_\delta w_0) = \sF_0(\beta_0, w_0) + I(\beta_0, \delta) = \mu + I(\beta_0, \delta). 
\end{equation}

Equations \eqref{eq:ladelta}, \eqref{id1}, \eqref{id2} \eqref{id3} imply that
\begin{equation} \label{eq:ladelta2}
\begin{split}
\la_\delta &= \sF_\delta(\beta_\delta - \cE_\delta \beta_0, w_\delta - \cE_\delta w_0) \\
&-\frac{E \sigma}{24 (1-\sigma^2)}\frac{\sigma \, d}{(1-\sigma) + \sigma d} \norma{\Div_x (\beta_\delta^I - \cE_\delta \beta_0)}^2_{\cH_\delta} \\
&+ \mu + I(\beta_0, \delta) + \int_{\Omega_\delta}\cE_\delta f_0 \cdot (\beta_\delta- \cE_\delta \beta_0, w_\delta - \cE_\delta w_0)^T \delta^{-d} dxdy \\
&+ 2 L(\beta_\delta, \beta_0, \delta)
\end{split}
\end{equation}
and, due to Remark \ref{rmk:quantitativeconv}, 
\[
|L(\beta_\delta, \beta_0, \delta)| \leq C_L \delta \norma{f_0}_{\cH_0} \norma{\Div \beta_0}_{H^1(\Omega)}.
\]
This concludes the proof of \eqref{eq:lowerbound}.\\
By \eqref{eq:upperbound} and \eqref{eq:lowerbound}, keeping into account \eqref{positivity}, we conclude that
\[
\begin{split}
&\sF^{\rm hom} _\delta(\beta_\delta - \cE_\delta \beta_0, w_\delta - \cE_\delta w_0)-\frac{E \sigma}{24 (1-\sigma^2)}\frac{\sigma \, d}{(1-\sigma) + \sigma d} \norma{\Div_x (\beta_\delta^I - \cE_\delta \beta_0)}^2_{\cH_\delta}\\
&\leq |2L(\beta_\delta, \beta_0, \delta)| \leq 2C_L \delta \norma{f_0}_{\cH_0} \norma{\Div \beta_0}_{H^1}.
\end{split}
\]
Since $\sF_\delta^{\rm hom}$ contains the square of the $\cH_\delta$-norm of $\Div_x(\beta_\delta^I - \cE_\delta \beta_0)$, a simple calculation show that there exists $C(\sigma) > 0$ such that
\[
\begin{split}
\sF_\delta^{\rm hom}(\beta_\delta  - \cE_\delta \beta_0, w_\delta - \cE_\delta w_0) -\frac{E \sigma}{24 (1-\sigma^2)}&\frac{\sigma \, d}{(1-\sigma) + \sigma d} \norma{\Div_x (\beta_\delta^I - \cE_\delta \beta_0)}^2_{\cH_\delta} \\
&\geq C(\sigma) \sF_\delta^{\rm hom}(\beta_\delta  - \cE_\delta \beta_0, w_\delta - \cE_\delta w_0).
\end{split}
\]
Therefore, 
\begin{equation} \label{energybound}
0 < C(\sigma) \sF_\delta^{\rm hom}(\beta_\delta  - \cE_\delta \beta_0, w_\delta - \cE_\delta w_0) \leq 2C_L \delta \norma{f_0}_{\cH_0} \norma{\Div \beta_0}_{H^1} \leq 2C_L C_{\rm reg} \delta \norma{f_0}_{\cH_0}^2,
\end{equation}
where in the last inequality we used that, due to elliptic regularity theory in the fixed limiting domain $\Omega$ there exists $C_{\rm reg} > 0$ such that
\[
\norma{\beta_0^I}_{H^2(\Omega)} \leq C_{\rm reg} \norma{f_0}_{L^2(\Omega)}.
\]

Just recall that 
$$\sF_\delta^{\rm hom}(\beta_\delta  - \cE_\delta \beta_0, w_\delta - \cE_\delta w_0)  \geq \min\bigg\{\frac{t^2}{24}, \frac{1}{2} \bigg\} \norma{(\beta_\delta  - \cE_\delta \beta_0, w_\delta - \cE_\delta w_0)}_{\cH_\delta}^2$$
to conclude that
\[
\norma{(\beta_\delta  - \cE_\delta \beta_0, w_\delta - \cE_\delta w_0)}_{\cH_\delta}^2 \leq 2 \bigg( C(\sigma) \min\bigg\{\frac{t^2}{24}, \frac{1}{2} \bigg\}\bigg)^{-1} C_L C_{\rm reg}\, \delta \norma{f_0}_{\cH_0}^2.
\]
\end{proof}

\begin{rem}
In the general case when $\Omega_\delta$ is as in \eqref{def:O_delta}, the upper bound $\la_\delta \leq \mu + I(\beta_0, \delta)$ still holds; however, the lower bound \eqref{eq:lowerbound} does not hold anymore because the identity $\p_{x_j} \cM_\delta(\beta_\delta) = \cM_\delta(\p_{x_j} \beta_\delta)$ is not valid when $g$ is not constant. To overcome this problem, a possible strategy is to show that the commutators $[\p_{x_j}, \cM_\delta]$ can be estimated with either quantities which are $O(\delta)$ as $\delta \to 0^+$ or with quantities that are already appearing in the lower bound in the flat case. Unfortunately, one has
\begin{equation} \label{eq:commutator}
|[\p_{x_j}, \cM_\delta] u|(x) \leq  C \delta \frac{1}{|\Omega_\delta^{II}(x)|}\int_{\Omega^{II}_\delta(x)} |\nabla_y u| + C \frac{1}{|\Omega_\delta^{II}(x)|}\int_{\Omega^{II}_\delta(x)} |u - \cM_\delta u| 
\end{equation}
for every $u \in H^1(\Omega_\delta)$. As remarked above, when $u = \beta_\delta^{I,i}$, we cannot proceed since do not have uniform apriori estimates on $\nabla_y \beta_\delta^{I,i}$.
\end{rem}

We show now that \eqref{rateconv} implies a rate of convergence for the eigenvalues. 

\begin{lemma} \label{lemma:projconv}
Let $\la_0 \in \sigma(A_0)$ be an eigenvalue of multiplicity $m$ and let $\la^j_\delta \in \sigma(A_\delta)$ be such that $\la^j_\delta \to \la_0$ as $\delta \to 0^+$, $j = 1, \dots, m$. Let $\gamma$ be a closed Jordan curve in $\C$ containing $\la_0$ and $\la^j_\delta$, $j = 1, \dots, m$, but no other point of $\sigma(A_0) \cup \sigma(A_\delta)$. Define
\[
P_\gamma^\delta := \frac{1}{2 \pi i} \int_\gamma (A_\delta - z)^{-1} d\gamma(z), \qquad P_\gamma^0 := \frac{1}{2\pi i} \int_\gamma (A_0 - z)^{-1} d\gamma(z).
\]
Then 
\[
\norma{P_\gamma^\delta \cE_\delta - \cE_\delta P_\gamma^0}_{\cL(\cH_0, \cH_\delta)} \leq C \omega(\delta).
\]
In particular, if $(u^j_\delta)_{j=1, \dots,m}$ is an orthonormal family in $L^2(\Omega_\delta)$ satisfying $A_\delta u^j_\delta = \la^j_\delta u^j_\delta$ and $A_0 u^0 = \la_0 u^0$, $\norma{u_0}_{\cH_0} = 1$, then
\[
\bigg \lVert \sum_{j=1}^m (u^j_\delta, \cE_\delta u_0) u^j_\delta - \cE_\delta u_0 \bigg \rVert \leq C \omega(\delta).
\] 
\end{lemma}
\begin{proof}
By definition of the projections $P_\gamma^\delta$, $P_\gamma^0$, it is immediate to check that
\[
\norma{P_\gamma^\delta \cE_\delta - \cE_\delta P_\gamma^0}_{\cL(\cH_0, \cH_\delta)} \leq \frac{|\gamma|}{2 \pi} \sup_{z \in \supp(\gamma)}\norma{(A_\delta - z)^{-1} \cE_\delta - \cE_\delta (A_0 - z)^{-1}}, 
\]
and since $\supp(\gamma) \subset \bigcap_{\delta \geq 0} \rho(A_\delta)$, by \eqref{rateconv} we conclude that 
\[
\norma{P_\gamma^\delta \cE_\delta - \cE_\delta P_\gamma^0}_{\cL(\cH_0, \cH_\delta)} \leq \frac{|\gamma|}{2 \pi} C \omega(\delta).
\]
\end{proof}

\begin{theorem} \label{thm:rateeigen}
Let $\la_0 \in \sigma(A_0)$ be an eigenvalue of multiplicity $m\geq 1$ and let $\la^i_\delta \in \sigma(A_\delta)$, $i =1, \dots, m$, be such that $\la^i_\delta \to \la_0$ as $\delta \to 0^+$. Let $\gamma$ be the Jordan curve containing $\la_\delta$ and $\la_0$ as in the previous lemma. Then
\begin{equation} \label{eq:eigenest}
\sum_{i=1}^m|\la^i_\delta - \la_0| \leq C \omega(\delta) |\la_0| \, \sup_{\delta \geq 0}\bigg[\norma{\cE_\delta P_\gamma^0}^{-1} \sum_{i=1}^m (1 + |\la^i_\delta| + |\la^i_\delta| \norma{A_\delta^{-1}})\bigg].
\end{equation}
\end{theorem}
\begin{proof}
We have the trivial identities
\[
A_\delta P_\gamma^\delta \cE_\delta = \sum_{i=1}^m \la^i_\delta P_\gamma^\delta \cE_\delta, \qquad  A_0 P_\gamma^0 = \la_0 P_\la^0,
\]
from which we deduce that
\[
P_\gamma^\delta \cE_\delta = \sum_{i=1}^m \la^i_\delta A_\delta^{-1} P_\gamma^\delta \cE_\delta, \qquad  \cE_\delta P_\gamma^0 = \la_0 \cE_\delta A_0^{-1} P_\la^0.
\]
In the sequel we omit the dependence on $\gamma$ when we write the projections. Subtracting the previous identities we get
\[
\begin{split}
P^\delta \cE_\delta - \cE_\delta P^0 &= \sum_{i=1}^m \la^i_\delta A_\delta^{-1} P^\delta \cE_\delta - \la_0 \cE_\delta A_0^{-1} P^0\\
&= \sum_{i=1}^m \big(\la^i_\delta A_\delta^{-1} \cE_\delta P^0 - \la_0 \cE_\delta A_0^{-1} P^0 + \la^i_\delta A_\delta^{-1} (P^\delta \cE_\delta - \cE_\delta P^0)\big) \\
&= \sum_{i=1}^m \big(\la^i_\delta (A_\delta^{-1} \cE_\delta - \cE_\delta A_0^{-1} ) P^0 + (\la^i_\delta - \la_0) \cE_\delta A_0^{-1} P^0 + \la^i_\delta A_\delta^{-1} (P^\delta \cE_\delta - \cE_\delta P^0)\big).
\end{split}
\]
Therefore
\[
\sum_{i=1}^m(\la_\delta^i - \la_0) \cE_\delta A_0^{-1} P^0 = \sum_{i=1}^m\big[(1 - \la^i_\delta A_\delta^{-1}) (P^\delta \cE_\delta - \cE_\delta P^0) - \la^i_\delta (A_\delta^{-1} \cE_\delta - \cE_\delta A_0^{-1} ) P^0\big].
\]
By Lemma \ref{lemma:projconv}, $\norma{P^\delta \cE_\delta - \cE_\delta P^0} \leq C \omega(\delta)$ and by \eqref{rateconv}, $\norma{A_\delta^{-1} \cE_\delta - \cE_\delta A_0^{-1}} \leq C \omega(\delta)$; moreover $A_0^{-1}P_0 = \frac{P_0}{\la_0}$. We then conclude that
\[
\sum_{i=1}^m |\la^i_\delta - \la_0 | \leq C \omega(\delta) |\la_0|\sum_{i=1}^m \bigg( \frac{\norma{(1 - \la^i_\delta A_\delta^{-1})} + |\la^i_\delta|}{\norma{\cE_\delta P^0}} \bigg),
\]
from which \eqref{eq:eigenest} follows easily.
\end{proof}

\begin{rem}
Since $\la^i_\delta \to \la_0$, there exists a constant $M \geq 1$ such that $|\la^i_\delta| \leq M |\la_0|$ for all $i = 1, \dots, m$. Furthermore, for the explicit choice of $(\cE_\delta)_\delta$ as in Theorem \ref{sp_conv}, $\norma{\cE_\delta} = 1$ for all $\delta$. We can then rewrite \eqref{eq:eigenest} in the simpler, yet less precise form,
\[
\sum_{i=1}^m|\la^i_\delta - \la_0| \leq C \omega(\delta) |\la_0| \, \frac{(1 + M |\la_0| + 2 M |\la_0| \norma{A_0^{-1}})}{\norma{P_\gamma^0}},
\]
so that the constant on the right-hand side does not depend on $\delta$.
\end{rem}


\section*{Acknowledgements} The authors are members of the Gruppo Nazionale per l’Analisi Matematica, la Probabilit\`a e le loro Applicazioni (GNAMPA) of the Istituto Nazionale di Alta Matematica (INdAM). The first author acknowledges support from the project ``Perturbation problems and asymptotics for elliptic differential equations: variational and potential theoretic methods" funded by the MUR Progetti di Ricerca di Rilevante Interesse Nazionale (PRIN) Bando 2022 grant 2022SENJZ3, and was also partially supported by the GNAMPA 2023 project ``Operatori differenziali e integrali in geometria spettrale''.\\
The second author gratefully acknowledges the funding of the European Union - NextGenerationEU, through the National Recovery and Resilience Plan (NRRP) of the Italian University and Research Ministry, e.INS project 00000038, CUP J83C21000320007, Spoke 06.
\section*{Data access statement}
All data are included in the text.

\bibliographystyle{abbrv}
\bibliography{RMb}

\end{document}